\documentclass{article}%

\usepackage{amsmath,amssymb,amsfonts}%
\usepackage{theorem}%
\usepackage{color}%
\usepackage{hyperref}%

\def\MR#1#2{\href{http://www.ams.org/mathscinet-getitem?mr=#1}{#2}}%

\theoremstyle{change}%
\newtheorem{definition}{Definition:}[section]%
\newtheorem{proposition}[definition]{Proposition:}%
\newtheorem{theorem}[definition]{Theorem:}%
\newtheorem{lemma}[definition]{Lemma:}%
\newtheorem{corollary}[definition]{Corollary:}%
{\theorembodyfont{\rmfamily}\newtheorem{remark}[definition]{Remark:}}%
{\theorembodyfont{\rmfamily}}%

\newenvironment{proof}
  {{\bf Proof:}}
  {\qquad \hspace*{\fill} $\Box$}%
  
\newcommand{\fa}{\mathfrak{a}}%
\newcommand{\fg}{\mathfrak{g}}%
\newcommand{\fk}{\mathfrak{k}}%
\newcommand{\fl}{\mathfrak{l}}%
\newcommand{\fm}{\mathfrak{m}}%
\newcommand{\fn}{\mathfrak{n}}%
\newcommand{\fp}{\mathfrak{p}}%
\newcommand{\fs}{\mathfrak{s}}%
\newcommand{\fz}{\mathfrak{z}}%

\newcommand{\Ad}{\operatorname{Ad}}%
\newcommand{\ad}{\operatorname{ad}}%
\newcommand{\tr}{\operatorname{tr}}%
\newcommand{\id}{\operatorname{id}}
\newcommand{\inner}{\operatorname{int}}%
\newcommand{\cl}{\operatorname{cl}}%
\newcommand{\fix}{\operatorname{fix}}%
\newcommand{\Sl}{\operatorname{Sl}}%
\newcommand{\tm}{\times}%
\newcommand{\ep}{\varepsilon}%

\newcommand{\rma}{\mathrm{\bf a}}%
\newcommand{\rmd}{\mathrm{d}}%
\newcommand{\rme}{\mathrm{e}}%
\newcommand{\rmh}{\mathrm{h}}%
\newcommand{\rmR}{\mathrm{\bf R}}%
\newcommand{\rmA}{\mathrm{\bf A}}%

\newcommand{\CC}{\mathcal{C}}%
\newcommand{\EC}{\mathcal{E}}%
\newcommand{\LC}{\mathcal{L}}%
\newcommand{\MC}{\mathcal{M}}%
\newcommand{\OC}{\mathcal{O}}%
\newcommand{\QC}{\mathcal{Q}}%
\newcommand{\RC}{\mathcal{R}}%
\newcommand{\SC}{\mathcal{S}}%
\newcommand{\UC}{\mathcal{U}}%
\newcommand{\VC}{\mathcal{V}}%
\newcommand{\WC}{\mathcal{W}}%

\newcommand{\B}{\mathbb{B}}%
\newcommand{\E}{\mathbb{E}}%
\newcommand{\N}{\mathbb{N}}%
\newcommand{\F}{\mathbb{F}}%
\newcommand{\R}{\mathbb{R}}%
\newcommand{\Z}{\mathbb{Z}}%
  
\begin{document}

\title{Hyperbolic Chain Control Sets on Flag Manifolds}%
\author{Adriano Da Silva\footnote{Imecc - Unicamp, Departamento de Matem\'atica, Rua S\'ergio Buarque de Holanda, 651, Cidade Universit\'aria Zeferino Vaz 13083-859, Campinas - SP, Brasil; ajsilvamat@hotmail.com}\ \ and Christoph Kawan\footnote{Courant Institute of Mathematical Sciences, New York University, 251 Mercer Street, New York City, USA; kawan@cims.nyu.edu}}%
\maketitle%

\begin{abstract}%
In this paper we study the chain control sets of right-invariant control systems on the flag manifolds of a non-compact semisimple Lie group. We prove that each chain control set is partially (skew-) hyperbolic over the associated control flow. Moreover, we characterize those chain control sets that are uniformly (skew-) hyperbolic in terms of the flag type of the control flow.%
\end{abstract}


\section{Introduction}%

A right-invariant control system on a Lie group $G$ is a control-affine system whose drift and control vector fields are invariant by right translations (cf.~Jurdjevic and Sussmann \cite{JSu} for an early study of such systems). Any system of this type induces control-affine systems on the homogeneous spaces $G/P$, where $P$ is a closed subgroup. In the case that $G$ is a non-compact semisimple group and $P=P_{\Theta}$ a parabolic subgroup, i.e., $\F_{\Theta} = G/P_{\Theta}$ is a generalized flag manifold, the controllability properties of such induced systems are well understood (cf., for instance, Barros and San Martin \cite{BSM}, San Martin \cite{SM1,SM2}, San Martin and Tonelli \cite{STo}). In particular, their control and chain control sets can be described algebraically in terms of the fixed point sets of certain group elements acting on $\F_{\Theta}$. These fixed point sets are the disjoint unions of finitely many connected components each of which is a compact submanifold. Moreover, the connected components are in bijection with a double coset space of the form $\WC_H\backslash\WC/\WC_{\Theta}$, where $\WC$ is the Weyl group of $G$ and $\WC_H$, $\WC_{\Theta}$ are certain subgroups depending on the element $H$ acting on $\F_{\Theta}$ and on the subset $\Theta$ of the simple roots, which defines the parabolic subgroup $P_{\Theta}$, respectively. As a result, also the control sets and the chain control sets are parametrized by double coset spaces of this form, where $H$ has to be chosen as a characteristic element depending on the given system. In this paper, we mainly concentrate on the chain control sets which, by a general result on control-affine systems (cf.~Colonius and Kliemann \cite{CKl}), are the projections of the chain recurrent components of the associated control flow $\phi_t$ on $\UC\tm\F_{\Theta}$, where $\UC$ is the set of admissible control functions on which the shift flow $\theta_tu = u(\cdot + t)$ acts. The projection of $\phi_t$ to $\F_{\Theta}$ is a cocycle $\varphi$ over $\theta$, defined by the solutions $\varphi(\cdot,x,u)$ of the differential equations constituting the control system.  We prove the existence of an invariant continuous decomposition%
\begin{equation*}
  T_x \F_{\Theta} = \SC_{\Theta,w}(u,x) \oplus \CC_{\Theta,w}(u,x) \oplus \UC_{\Theta,w}(u,x)%
\end{equation*}
for each $(u,x)$ in the chain recurrent component over the chain control set associated to a Weyl group element $w$ (or the corresponding double coset, respectively). Here we have uniform contraction on $\SC_{\Theta,w}$ and uniform expansion on $\UC_{\Theta,w}$, while on $\CC_{\Theta,w}$ directions corresponding to zero Lyapunov exponents exist. In particular, we are interested in the case when $\CC_{\Theta,w}$ is trivial, i.e., the case of uniform hyperbolicity. As it turns out, this happens if and only if $\langle\Theta(\phi)\rangle\subset w\langle\Theta\rangle$, where $\Theta(\phi)$ is a characteristic subset of the simple roots, depending on the control flow $\phi$, the so-called \emph{flag type} of the flow. By a general result proved in Colonius and Du \cite{CDu}, under appropriate regularity assumptions, in this case the chain control set is the closure of a control set. Our main motivation to characterize the hyperbolic chain control sets is the fact that there exist good estimates for their invariance entropy (as introduced in Colonius and Kawan \cite{CKa}, see also the monograph Kawan \cite{Kaw}), and our aim is to use the characterization worked out in this paper to provide a formula for their entropy in a future work.%

The paper is structured as follows: In Section \ref{sec_dynctrl} we recall basic facts about dynamical and control systems, and in Section \ref{sec_semisimple} about semisimple Lie groups and their flag manifolds. Section \ref{sec_flagbundleflows} describes the chain recurrent components for flows of automorphisms on flag bundles and introduces the so-called $\fa$-cocycle over the flow on the maximal flag bundle, which is a vector-valued additive cocycle with values in the maximal abelian subspace associated with the Lie algebra of the structural group. Furthermore, two associated cocycles for the flows on the partial flag bundles are introduced, which are used to describe the unstable and stable determinants for the flow on the hyperbolic chain control sets. Finally, in Section \ref{sec_hyperbolic}, it is proved that each chain control set has a partially hyperbolic structure (in a loose sense) and the hyperbolic chain control sets are characterized. Furthermore, an independent proof of the fact that a hyperbolic chain control set is the closure of a control set is given for the flag case.%

Some remarks on notation: We write $\Z$, $\N$ and $\R$ for the sets of integers, positive integers and real numbers, respectively, and $\R^d$ for the $d$-dimensional Euclidean space.
If $V$ is a finite-dimensional real vector space, $V^*$ denotes its dual space, the space of real-valued linear functionals on $V$. If $M$ is a smooth manifold, we denote by $T_xM$ the tangent space at $x\in M$, and by $TM$ the tangent bundle of $M$. If $\varphi:M\rightarrow N$ is a differentiable map between smooth manifolds, we write $(\rmd\varphi)_x:T_xM \rightarrow T_{\varphi(x)}N$ for its derivative at $x\in M$. Furthermore, we write $\cl A$ for the closure of a set (in a metric space), and $\inner A$ for its interior. If $\phi$ is a continuous flow on a metric space $X$, we write $\omega(x)$ and $\alpha(x)$ for the $\omega$- and $\alpha$-limit set of a point $x\in X$, i.e., the sets of limit points of the forward and backward trajectories, respectively.%

\section{Dynamical and Control Systems}\label{sec_dynctrl}%

In this section, we recall well-known facts about flows on metric spaces and control-affine systems that can be found, e.g., in Colonius and Kliemann \cite{CKl}.%

\subsection{Morse Decompositions}%

Let $\phi:\R\tm X\rightarrow X$, $(t,x)\mapsto\phi_t(x)$, be a continuous flow on a compact metric space $(X,d)$. A compact set $K\subset X$ is called \emph{isolated invariant} if it is invariant, i.e., $\phi_t(K)\subset K$ for all $t\in\R$, and if there is a neighborhood $N$ of $K$ such that the implication%
\begin{equation*}
  \phi_t(x) \in N \mbox{\ for all\ } t\in\R \quad\Rightarrow\quad x\in K%
\end{equation*}
holds. A \emph{Morse decomposition} of $\phi$ is a finite collection $\{\MC_1,\ldots,\MC_n\}$ of nonempty pairwise disjoint isolated invariant compact sets satisfying:%
\begin{enumerate}
\item[(A)] For all $x\in X$, the $\alpha$- and $\omega$-limit sets $\alpha(x)$ and $\omega(x)$, respectively, are contained in $\bigcup_{i=1}^n\MC_i$.%
\item[(B)] Suppose there are $\MC_{j_0},\ldots,\MC_{j_l}$ and $x_1,\ldots,x_l\in X\backslash\bigcup_{i=1}^n\MC_i$ with%
\begin{equation*}
  \alpha(x_i) \in \MC_{j_{i-1}} \mbox{\quad and\quad } \omega(x_i) \in \MC_{j_i}%
\end{equation*}
for $i=1,\ldots,l$. Then $\MC_{j_0}\neq\MC_{j_l}$.%
\end{enumerate}
The elements of a Morse decomposition are called \emph{Morse sets}. We say that a compact invariant set $A$ is an \emph{attractor} if it admits a neighborhood $N$ such that $\omega(N)=A$. A \emph{repeller} is a compact invariant set $R$ which has a neighborhood $N^*$ with $\alpha(N^*) = R$. A Morse decomposition is \emph{finer} than another one if every element of the second one contains one of the first.%

Morse decompositions are related to the chain recurrent set of the flow. Recall that an \emph{$(\ep,T)$-chain} from $x\in X$ to $y\in X$ is given by $n\in\N$, points $x_0,x_1,\ldots,x_n \in X$ with $x_0 = x$ and $x_n=y$, and times $T_0,\ldots,T_{n-1}\geq T$ such that $d(\phi_{T_i}(x_i),x_{i+1})<\ep$ for $i=0,1,\ldots,n-1$. A subset $Y\subset X$ is \emph{chain transitive} if for all $x,y\in Y$ and $\ep,T>0$ there exists an $(\ep,T)$-chain from $x$ to $y$. A point $x\in X$ is \emph{chain recurrent} if for all $\ep,T>0$ there exists an $(\ep,T)$-chain from $x$ to $x$. The \emph{chain recurrent set $\RC = \RC(\phi)$} is the set of all chain recurrent points. Then the connected components of $\RC$ coincide with the maximal chain transitive subsets, which are also called the \emph{chain recurrent components} of $\phi$. A finest Morse decomposition for $\phi$ exists iff there are only finitely many chain recurrent components. In this case, the Morse sets coincide with the chain recurrent components (cf.~also \cite[Thm.~3.17]{PSM} for this result in a more general context, where the space $X$ is only assumed to be compact Hausdorff).%

\subsection{Control-Affine Systems}\label{subsec_cas}%

A \emph{control-affine system} is given by a family%
\begin{equation}\label{eq_cas}
  \dot{x}(t) = f_0(x(t)) + \sum_{i=1}^mu_i(t)f_i(x(t)),\quad u\in\UC,%
\end{equation}
of ordinary differential equations. Here $f_0,f_1,\ldots,f_m$ are $\CC^k$-vector fields ($k\geq1$) on a finite-dimensional smooth manifold $M$, the \emph{state space} of the system. $f_0$ is called the \emph{drift vector field} and $f_1,\ldots,f_m$ the \emph{control vector fields}. We assume that the set $\UC$ of \emph{admissible control functions} is given by%
\begin{equation*}
  \UC = \left\{u:\R\rightarrow\R^m\ :\ u \mbox{ is measurable with } u(t)\in U \mbox{ a.e.}\right\},%
\end{equation*}
where $U\subset\R^m$ is a compact and convex set with $0\in\inner U$. For each $u\in\UC$ the corresponding (Carath\'eodory) differential equation \eqref{eq_cas} has a unique solution $\varphi(t,x,u)$ with initial value $x = \varphi(0,x,u)$. The systems considered in this paper all have globally defined solutions, which give rise to a map%
\begin{equation*}
  \varphi:\R \tm M \tm \UC \rightarrow M,\quad (t,x,u) \mapsto \varphi(t,x,u),%
\end{equation*}
called the \emph{transition map} of the system. We also use the notation $\varphi_{t,u}:M\rightarrow M$ for the map $x\mapsto\varphi(t,x,u)$. If the vector fields $f_0,f_1,\ldots,f_m$ are of class $\CC^k$, then $\varphi$ is of class $\CC^k$ with respect to the state variable and the corresponding partial derivatives of order $1$ up to $k$ depend continuously on $(t,x,u) \in \R\tm M\tm\UC$ (cf.~\cite[Thm.~1.1]{Kaw}).%

The transition map $\varphi$ is a cocycle over the shift flow%
\begin{equation*}
  \theta:\R \tm \UC \rightarrow \UC,\quad (t,u) \mapsto \theta_tu = u(\cdot + t),%
\end{equation*}
i.e., it satisfies%
\begin{equation*}
  \varphi(t+s,x,u) = \varphi(s,\varphi(t,x,u),\theta_tu) \mbox{\quad for all\ } t,s\in\R,\ x\in M,\ u\in\UC.%
\end{equation*}
Together with the shift flow, $\varphi$ constitutes a continuous skew-product flow%
\begin{equation*}
  \phi:\R \tm \UC \tm M \rightarrow \UC \tm M,\quad (t,u,x) \mapsto (\theta_tu,\varphi(t,x,u)),%
\end{equation*}
where $\UC$ is endowed with the weak$^*$-topology of $L^{\infty}(\R,\R^m) = L^1(\R,\R^m)^*$, which gives $\UC$ the structure of a compact metrizable space. The flow $\phi$ is called the \emph{control flow} of the system (cf.~\cite{CKl,Kaw}). The base flow $\theta$ is chain transitive.%

We write%
\begin{equation*}
  \OC^+_{\leq\tau}(x) := \left\{y\in M\ |\ \exists t\in[0,\tau],\ u\in\UC:\ \varphi(t,x,u) = y\right\}%
\end{equation*}
for the \emph{set of points reachable from $x$ up to time $\tau$}. The \emph{positive orbit} of $x$ is given by%
\begin{equation*}
  \OC^+(x) := \bigcup_{\tau>0}\OC^+_{\leq\tau}(x).%
\end{equation*}
Analogously, we define the \emph{set of points controllable to $x$ within time $\tau$} by%
\begin{equation*}
  \OC^-_{\leq\tau}(x) := \left\{y\in M\ |\ \exists t\in[0,\tau],\ u\in\UC:\ \varphi(t,y,u) = x\right\},%
\end{equation*}
and the \emph{negative orbit of $x$} by%
\begin{equation*}
  \OC^-(x) := \bigcup_{\tau>0}\OC^-_{\leq\tau}(x).%
\end{equation*}

In the following, we fix a metric $d$ on $M$ (not necessarily a Riemannian distance). A set $D\subset M$ is called \emph{controlled invariant (in forward time)} if for each $x\in D$ there exists $u\in\UC$ with $\varphi(t,x,u)\in D$ for all $t\geq0$. It is called a \emph{control set} if it satisfies the following properties:%
\begin{enumerate}
\item[(A)] $D$ is controlled invariant.%
\item[(B)] Approximate controllability holds on $D$, i.e., $D\subset\cl\OC^+(x)$ for all $x\in D$.%
\item[(C)] $D$ is maximal (w.r.t.~set inclusion) with the properties (A) and (B).%
\end{enumerate}

A set $E\subset M$ is called \emph{full-time controlled invariant} if for each $x\in E$ there exists $u\in\UC$ with $\varphi(\R,x,u)\subset E$. The \emph{full time lift} $\EC$ of $E$ is defined by%
\begin{equation*}
  \EC = \mathrm{Lift}(E) := \left\{ (u,x) \in \UC \tm M\ :\ \varphi(\R,x,u) \subset E \right\},%
\end{equation*}
which is easily seen to be compact and $\phi$-invariant. For points $x,y\in M$ and numbers $\ep,\tau>0$, a \emph{controlled $(\ep,\tau)$-chain from $x$ to $y$} is given by $n\in\N$, points $x_0,\ldots,x_n\in M$, control functions $u_0,\ldots,u_{n-1}\in\UC$, and times $t_0,\ldots,t_{n-1}\geq\tau$ such that $x_0 = x$, $x_n = y$, and $d(\varphi(t_i,x_i,u_i),x_{i+1}) < \ep$ for $i=0,\ldots,n-1$. A set $E\subset M$ is called a \emph{chain control set} if it satisfies the following properties:%
\begin{enumerate}
\item[(A)] $E$ is full-time controlled invariant%
\item[(B)] For all $x,y\in E$ and $\ep,\tau>0$ there exists an $(\ep,\tau)$-chain from $x$ to $y$ in $M$.%
\item[(C)] $E$ is maximal (w.r.t.~set inclusion) with the properties (A) and (B).%
\end{enumerate}
Every chain control set is closed, while this is not the case for control sets. Moreover, every control set with nonempty interior is contained in a chain control set if local accessibility holds, and we have the following result.%

\begin{proposition}\label{prop_ccs}
If $M$ is compact, then the full-time lift $\EC$ of a chain control set $E$ is a chain recurrent component of the control flow $\phi$. Conversely, if $\EC\subset\UC\tm M$ is a chain recurrent component of $\phi$, then the projection%
\begin{equation*}
  E = \left\{x\in M\ :\ \exists u\in\UC \mbox{ with } (u,x)\in\EC\right\}%
\end{equation*}
of $\EC$ to $M$ is a chain control set.%
\end{proposition}

System \eqref{eq_cas} is called \emph{locally accessible at $x$} provided that for all $\tau>0$ the sets $\OC^+_{\leq\tau}(x)$ and $\OC^-_{\leq\tau}(x)$ have nonempty interiors. It is called \emph{locally accessible} if it is locally accessible at every point $x\in M$. If the vector fields $f_0,f_1,\ldots,f_m$ are of class $\CC^{\infty}$, the \emph{Lie algebra rank condition} (Krener's criterion) guarantees local accessibility: Let $\LC=\LC(f_0,f_1,\ldots,f_m)$ denote the smallest Lie algebra of vector fields on $M$ containing $f_0,f_1,\ldots,f_m$. If $\LC(x) := \{f(x) : f\in\LC\} = T_xM$, then the system is locally accessible at $x$. If $f_0,f_1,\ldots,f_m$ are analytic vector fields, the criterion is also necessary.%

In the following, we assume that $M$ is endowed with a Riemannian metric. We call a compact full-time controlled invariant set $Q$ \emph{(uniformly) hyperbolic} if for each $(u,x)\in\QC = \mathrm{Lift}(Q)$ there exists a decomposition%
\begin{equation*}
  T_xM = E^-_{u,x} \oplus E^+_{u,x}%
\end{equation*}
such that the following properties hold:%
\begin{enumerate}
\item[(H1)] $(\rmd\varphi_{t,u})_x E^{\pm}_{u,x} = E^{\pm}_{\phi_t(u,x)}$ for all $t\in\R$ and $(u,x)\in\QC$.%
\item[(H2)] There exist constants $c,\lambda>0$ such that for all $(u,x)\in\QC$ we have%
\begin{equation*}
  \left\|(\rmd\varphi_{t,u})_x v\right\| \leq c^{-1}\rme^{-\lambda t}\|v\| \mbox{\quad for all\ } t\geq0,\ v\in E^-_{u,x},%
\end{equation*}
and%
\begin{equation*}
  \left\|(\rmd\varphi_{t,u})_x v\right\| \geq c\rme^{\lambda t}\|v\| \mbox{\quad for all\ } t\geq0,\ v\in E^+_{u,x}.%
\end{equation*}
\item[(H3)] The linear subspaces $E^{\pm}_{u,x}$ depend continuously on $(u,x)$, i.e., the projections $\pi^{\pm}_{u,x}:T_xM \rightarrow E^{\pm}_{u,x}$ along $E^{\mp}_{u,x}$ depend continuously on $(u,x)$.%
\end{enumerate}
Just as for classical hyperbolic sets (of autonomous dynamical systems), it can be shown that (H3) actually follows from (H1) and (H2). In particular, the subspaces $E^{\pm}_{u,x}$ are the fibers of subbundles $E^{\pm}\rightarrow\QC$ of the vector bundle%
\begin{equation*}
  \bigcup_{(u,x)\in\QC}\{u\}\tm T_xM \rightarrow \QC,\quad (u,v) \mapsto (u,\pi_{TM}(v)),%
\end{equation*}
with the base point projection $\pi_{TM}:TM \rightarrow M$. (cf.~also \cite[Sec.~6.3]{Kaw}).%

\section{Semisimple Lie Theory}\label{sec_semisimple}%

In this section, we recall basic facts on semisimple Lie groups and their flag manifolds. Throughout the paper, we will only consider connected non-compact semisimple groups $G$ with finite center. With regard to the actions of these groups on their flag manifolds, the assumption of a finite center causes no loss of generality, since the flag manifolds can be represented as orbits in Grassmann manifolds of subspaces of the Lie algebra $\fg$ under the action of $\Ad(G) \cong G/Z(G)$ and $G$ also acts on them by the adjoint action.%

Standard references for the theory of semisimple Lie groups and their flag manifolds are Duistermaat, Kolk and Varadarajan \cite{DKV}, Helgason \cite{Hel} and Warner \cite{War}.%

\subsection{Semisimple Lie Groups}

Let $G$ be a connected semisimple non-compact Lie group $G$ with finite center. The associated Lie algebra is denoted by $\fg$. We choose a Cartan involution $\theta:\fg\rightarrow\fg$, which yields the inner product $B_{\theta}(X,Y) = -C(X,\theta Y)$, where $C(X,Y) = \tr(\ad(X)\ad(Y))$ is the Cartan-Killing form. Since $\theta$ is self-adjoint with respect to $B_{\theta}$ and $\theta^2 = \id$, we have the Cartan decomposition%
\begin{equation*}
  \fg = \fk \oplus \fs,%
\end{equation*}
where $\fk$ is the $1$-eigenspace and $\fs$ the $(-1)$-eigenspace of $\theta$. Note that $\fk$ is a subalgebra of $\fg$, but $\fs$ only a linear subspace. We also have an associated Cartan decomposition $G = KS$ of the group with $K = \exp\fk$ and $S = \exp\fs$. The endomorphisms $\ad(X)$, $X\in\fk$, are skew-symmetric w.r.t.~$B_{\theta}$ and hence the automorphisms $\Ad(k)$, $k\in K$, are isometries. Let $\fa \subset \fs$ be a maximal abelian subspace, and for each $\alpha\in\fa^*$ put%
\begin{equation*}
  \fg_{\alpha} := \left\{X \in \fg\ :\ \ad(H)X = \alpha(H)X,\ \forall H\in\fa\right\}.%
\end{equation*}
The endomorphisms $\ad(H)$, $H\in\fa$, are self-adjoint with respect to $B_{\theta}$ and they commute. Hence, we can diagonalize them simultaneously and the nontrivial $\fg_{\alpha}$'s are the associated eigenspaces. The set%
\begin{equation*}
  \Pi := \left\{\alpha\in\fa^*\backslash\{0\} :\ \fg_{\alpha} \neq 0\right\}%
\end{equation*}
is called the \emph{set of roots of $G$}. The associated spaces $\fg_{\alpha}$, $\alpha\in\Pi$, are called \emph{root spaces}. The set of \emph{regular elements of $\fa$} is given by%
\begin{equation*}
  \left\{ H\in\fa\ : \ \alpha(H)\neq0,\ \forall \alpha\in\Pi \right\}.%
\end{equation*}
The connected components of this set, called \emph{Weyl chambers}, are convex cones. Choosing an (arbitrary) Weyl chamber $\fa^+$, it makes sense to define%
\begin{equation*}
  \alpha > 0 \qquad :\Leftrightarrow \qquad \alpha|_{\fa^+} > 0,\quad \alpha\in\Pi,%
\end{equation*}
since $\alpha\in\Pi$ cannot change its sign on a Weyl chamber. Then the \emph{sets of positive and negative roots}, respectively, are%
\begin{equation*}
  \Pi^+ := \left\{\alpha\in\Pi\ :\ \alpha>0\right\} \mbox{\quad and \quad} \Pi^- := -\Pi^+,%
\end{equation*}
and we have a disjoint union $\Pi = \Pi^+ \cup \Pi^-$. Moreover, we define the nilpotent subalgebras%
\begin{equation*}
  \fn^{\pm} := \sum_{\alpha\in\Pi^{\pm}}\fg_{\alpha},%
\end{equation*}
which yield the \emph{Iwasawa decomposition} of the Lie algebra:%
\begin{equation*}
  \fg = \fk \oplus \fa \oplus \fn^{\pm} \mbox{\quad (either } \fn^+ \mbox{ or } \fn^-).%
\end{equation*}
To each of these subalgebras there belongs a connected Lie subgroup of $G$, denoted by $K$, $A$ and $N^{\pm}$, respectively. This gives the Iwasawa decomposition $G = KAN^{\pm}$ of the group. The \emph{Weyl group} $\WC$ of $G$ is the group generated by the orthogonal reflections at the hyperplanes $\ker\alpha$, $\alpha\in\Pi$, or alternatively the quotient $M^*/M$, where $M^*$ and $M$ are the normalizer and the centralizer of $\fa$ in $K$, respectively, i.e., $M^* = \{k\in K : \Ad(k)\fa=\fa\}$ and $M = \{k\in K : \Ad(k)H=H,\ \forall H\in\fa\}$. A corresponding group isomorphism is obtained by associating to $kM \in M^*/M$ the automorphism $\Ad(k)|_{\fa}$ of $\fa$. The Weyl group acts simply transitively on the Weyl chambers. It also acts on the roots by%
\begin{equation*}
  w\alpha(H) = \alpha(w^{-1}H),\quad \forall H\in\fa,\ \alpha\in\Pi.%
\end{equation*}
We let $\Sigma \subset \Pi^+$ denote the set of those positive roots which cannot be written as linear combinations of other positive roots, and we call $\Sigma$ the \emph{set of simple roots}, which forms a basis of $\fa^*$. There exists a unique element $w_0\in\WC$, called the \emph{principal involution}, which takes $\Sigma$ to $-\Sigma$.%

\subsection{Flag Manifolds}%

Let $\Theta \subset \Sigma$ be an arbitrary subset. Then we write $\langle\Theta\rangle$ for the set of roots which are linear combinations (over $\Z$) of elements in $\Theta$. Moreover, we put%
\begin{equation*}
  \fa(\Theta) := \langle H_{\alpha}\ : \ \alpha\in\Theta \rangle,%
\end{equation*}
where $H_{\alpha}\in\fa$ is the coroot of $\alpha$, defined by $B_{\theta}(H_{\alpha},H) = \alpha(H)$, $H\in\fa$. Then $\fg(\Theta)$ is defined as the (semisimple) subalgebra generated by $\fa(\Theta) \oplus \sum_{\alpha\in\langle\Theta\rangle}\fg_{\alpha}$. We put $\fk(\Theta) := \fk \cap \fg(\Theta)$ and $\fn^{\pm}(\Theta) := \fn^{\pm} \cap \fg(\Theta)$. Then $\fg(\Theta)$ is the Lie algebra of a semisimple Lie group $G(\Theta)\subset G$, and $\fg(\Theta) = \fk(\Theta)\oplus\fa(\Theta)\oplus\fn^{\pm}(\Theta)$ is an Iwasawa decomposition of $\fg(\Theta)$, while $\Theta$ is the corresponding set of simple roots. We write $K(\Theta)$ for the connected Lie subgroup with Lie algebra $\fk(\Theta)$ and $A(\Theta) = \exp\fa(\Theta)$, $N^{\pm}(\Theta) = \exp\fn^{\pm}(\Theta)$ (which are also connected subgroups). Then $G(\Theta) = K(\Theta)A(\Theta)N^{\pm}(\Theta)$ is an Iwasawa decomposition of $G(\Theta)$. Let%
\begin{equation*}
  \fa_{\Theta} := \left\{ H\in\fa\ :\ \alpha(H)=0,\ \forall \alpha \in \Theta \right\}%
\end{equation*}
be the orthogonal complement of $\fa(\Theta)$ and note that%
\begin{equation*}
  \fa(\Theta_1 \cap \Theta_2) = \fa(\Theta_1) \cap \fa(\Theta_2) \mbox{\quad and\quad } \fa_{\Theta_1 \cap \Theta_2} = \fa_{\Theta_1} + \fa_{\Theta_2}.%
\end{equation*}
The subset $\Theta$ singles out a subgroup $\WC_{\Theta}$ of $\WC$ consisting of those elements which act trivially on $\fa_{\Theta}$. Alternatively, $\WC_{\Theta}$ can be defined as the subgroup generated by the reflections at the hyperplanes $\ker\alpha$, $\alpha\in\Theta$. Then $\WC_{\Theta}$ is isomorphic to the Weyl group $\WC(\Theta)$ of $G(\Theta)$. We let $Z_{\Theta}$ denote the centralizer of $\fa_{\Theta}$ in $G$ and $K_{\Theta} = Z_{\Theta} \cap K$, which implies%
\begin{equation}\label{eq_kthetaintersec}
  K_{\Theta_1 \cap \Theta_2} = K_{\Theta_1} \cap K_{\Theta_2}.%
\end{equation}
An Iwasawa decomposition of $Z_{\Theta}$ (which is a reductive Lie group) is given by%
\begin{equation}\label{eq_ztheta_iwasawa}
  Z_{\Theta} = K_{\Theta}AN^{\pm}(\Theta).%
\end{equation}
The \emph{parabolic subalgebra of type $\Theta\subset\Sigma$} is defined by%
\begin{equation*}
  \fp_{\Theta} := \fn^-(\Theta) \oplus \fp,\quad \fp := \fm \oplus \fa \oplus \fn^+,%
\end{equation*}
where $\fm$ is the Lie algebra of $M$, the centralizer of $\fa$ in $\fk$, or respectively, the part of the common $0$-eigenspace of the maps $\ad(H)$, $H\in\fa$, contained in $\fk$. The associated \emph{parabolic subgroup} $P_{\Theta}$ is the normalizer of $\fp_{\Theta}$ in $G$. Then $\fp_{\Theta}$ is the Lie algebra of $P_{\Theta}$.%

The empty set $\Theta=\emptyset$ yields the minimal parabolic subalgebra $\fp_{\emptyset} = \fp$. An Iwasawa decomposition of the associated subgroup $P$ is $P = MAN^+$. A corresponding decomposition, called \emph{Langlands decomposition}, of $P_{\Theta}$ is given by%
\begin{equation}\label{eq_langlands}
  P_{\Theta} = K_{\Theta}AN^+.%
\end{equation}
The \emph{flag manifold of type $\Theta$} is the $\Ad(G)$-orbit%
\begin{equation*}
  \F_{\Theta} := \Ad(G)\fp_{\Theta}%
\end{equation*}
with base point $b_{\Theta} := \fp_{\Theta}$ in the Grassmann manifold of $(\dim\fp_{\Theta})$-dimensional subspaces of $\fg$. In case $\Theta=\emptyset$ we also write $b_0 = b_{\Theta}$ and $\F = \F_{\emptyset}$ (the maximal flag manifold). Since the isotropy group of $b_{\Theta}$ is the subgroup $P_{\Theta}$, $\F_{\Theta}$ can be identified with the homogeneous space $G/P_{\Theta}$ via $\Ad(g)\fp_{\Theta}\mapsto gP_{\Theta}$. If $\Theta_1 \subset \Theta_2$, then $P_{\Theta_1}\subset P_{\Theta_2}$ and the projection%
\begin{equation*}
  \pi^{\Theta_1}_{\Theta_2}:\F_{\Theta_1} \rightarrow \F_{\Theta_2},\quad gP_{\Theta_1} \mapsto gP_{\Theta_2},%
\end{equation*}
is well-defined. In case $\Theta_1=\emptyset$, we just write $\pi_{\Theta_2}$ for this map.%

Note that the choice of the above subalgebras and subgroups is not unique. In fact, by conjugation with an element $k\in K$ one obtains a new set of subalgebras and subgroups. For instance, $\Ad(k)\fa$ is another maximal abelian subspace of $\fs$ and $\Ad(k)\fa^+$ an associated positive Weyl chamber. Also $\fg = \fk \oplus \Ad(k)\fa \oplus \Ad(k)\fn^{\pm}$ is another Iwasawa decomposition. We will use these conjugated settings frequently.%

Alternatively, one can define a flag manifold by choice of an element $H\in\cl\fa^+$. One puts%
\begin{equation*}
  \Theta(H) := \{\alpha\in\Sigma\ :\ \alpha(H) = 0\}%
\end{equation*}
and denotes the above subalgebras and groups replacing $\Theta$ by $H$, for instance, $\fp_H := \fp_{\Theta(H)}$. Conversely, given a subset $\Theta\subset\Sigma$, there exists a \emph{characteristic element} $H_{\Theta} \in \cl\fa^+$ such that $\Theta = \Theta(H_{\Theta})$. We let $Z_H$, $K_H$, and $\WC_H$ denote the centralizer of $H$ in $G$, $K$ and $\WC$, respectively. Furthermore, we introduce%
\begin{equation*}
  \fn^{\pm}_{\Theta} := \sum_{\alpha\in\Pi^{\pm}\backslash\langle\Theta\rangle}\fg_{\alpha}%
\end{equation*}
and note that $N^{\pm}(\Theta)$ centralizes $\fn^{\mp}_{\Theta}$.%

Finally, we briefly describe the construction of a $K$-invariant Riemannian metric on $\F_{\Theta}$. We have that $\fg = \fn^-_{\Theta} \oplus \fp_{\Theta}$. The group $K_{\Theta}$ is the isotropy subgroup of $b_{\Theta}$ in $K$, $\fn^-_{\Theta}$ is a $K_{\Theta}$-invariant subspace, and the restriction of $B_{\theta}$ to $\fn^-_{\Theta}$ is a $K_{\Theta}$-invariant inner product. By a standard construction, this allows to define a $K$-invariant metric on $\F_{\Theta}$ such that%
\begin{equation}\label{eq_kinvmetric}
  \langle (\rmd\pi_{\Theta})_1Y,(\rmd\pi_{\Theta})_1 Z \rangle_{b_{\Theta}} = B_{\theta}(Y,Z),\quad \forall Y,Z\in\fn^-_{\Theta}.% 
\end{equation}

\section{Flows on Flag Bundles}\label{sec_flagbundleflows}%

Let $G$ be a connected semisimple non-compact Lie group $G$ with finite center. We choose a Cartan involution $\theta$, a maximal abelian subspace $\fa$, and a positive Weyl chamber $\fa^+$. An element of $\fg$ of the form $Y = \Ad(g)H$ with $g\in G$ and $H\in\cl\fa^+$ is called a \emph{split element}. Analogously, an element of $G$ of the form $ghg^{-1}$ with $h\in\cl A^+$, $A^+ = \exp\fa^+$, is called a split element. The flow $\exp(tH)$ induced by a split element $H\in\cl\fa^+$ on $\F_{\Theta}$, is given by%
\begin{equation*}
  (t,\Ad(g)\fp_{\Theta}) \mapsto \Ad(\rme^{tH}g)\fp_{\Theta}.%
\end{equation*}
The associated vector field can be shown to be a gradient vector field with respect to an appropriate Riemannian metric on $\F_{\Theta}$. The connected components of the fixed point set of this flow are given by%
\begin{equation*}
  \fix_{\Theta}(H,w) = Z_H \cdot wb_{\Theta} = K_H \cdot wb_{\Theta},\quad w \in \WC.%
\end{equation*}
The sets $\fix_{\Theta}(H,w)$ are easily shown to be in bijection with the double coset space $\WC_H\backslash \WC/\WC_{\Theta}$. Each component $\fix_{\Theta}(H,w)$ is a compact and connected submanifold of $\F_{\Theta}$. If $H \in \Ad(k)\cl\fa^+$ for some $k\in K$, we have to work with the conjugated Iwasawa decomposition and hence in this case%
\begin{equation*}
  \fix_{\Theta}(H,w) = Z_H \cdot \Ad(k)wb_{\Theta}.%
\end{equation*}

We consider now a $G$-principal bundle $\pi:Q\rightarrow X$ over a compact metric space $X$ and an Iwasawa decomposition $G=KAN^{\pm}$. Then $G$ acts continuously from the right on $Q$, this action preserves the fibers, and is free and transitive on each fiber. In particular, this implies that each fiber is homeomorphic to $G$. An automorphism of $Q$ is a homeomorphism $\phi:Q\rightarrow Q$ which maps fibers to fibers and respects the right action of $G$ in the sense that $\phi(q \cdot g) = \phi(q) \cdot g$. For each set $\Theta\subset\Sigma$ of simple roots there exists a flag bundle $\E_{\Theta} = Q \tm_G \F_{\Theta}$ with typical fiber $\F_{\Theta}$ given by $(Q\tm \F_{\Theta})/\!\sim$, where $(q_1,b_1)\sim (q_2,b_2)$ iff there exists $g\in G$ with $q_1 = q_2 \cdot g$ and $b_1 = g^{-1} \cdot b_2$. The maximal flag bundle $Q\tm_G \F$ is denoted by $\E$.%

Now let $\phi_t:Q\rightarrow Q$ be a flow of automorphisms whose base flow on $X$ is chain transitive. This flow induces a flow on each of the associated flag bundles $\E_{\Theta}$, which we also denote by $\phi_t$. In the following subsection, we recall different characterizations of the chain recurrent components of these flows.%

\subsection{Morse Sets}%

Each of the flows $\phi_t:\E_{\Theta}\rightarrow\E_{\Theta}$ has finitely many chain recurrent components and thus admits a finest Morse decomposition. The Morse sets on $\E_{\Theta}$ can be described as follows (cf.~\cite[Thm.~9.11]{BSM} or \cite[Thm.~5.2]{SMS}).%

\begin{theorem}\label{thm_morsesets}
The following assertions hold:%
\begin{enumerate}
\item[(i)] There exist $H_{\phi}\in\cl\fa^+$ and a continuous $\phi$-invariant map%
\begin{equation*}
  h_{\phi}:Q \rightarrow \Ad(G)H_{\phi},\quad h_{\phi}(\phi_t(q)) \equiv h_{\phi}(q),%
\end{equation*}
into the adjoint orbit of $H_{\phi}$, which is equivariant, i.e., $h_{\phi}(q\cdot g) = \Ad(g^{-1})h_{\phi}(q)$, $q\in Q$, $g\in G$. The induced flow on $\E_{\Theta}$ admits a finest Morse decomposition whose elements are given fiberwise by%
\begin{equation*}
  \MC_{\Theta}(w)_{\pi(q)} = q \cdot \fix_{\Theta}(h_{\phi}(q),w),\quad w\in\WC.%
\end{equation*}
The set $\Theta(\phi) := \Theta(H_{\phi}) = \{\alpha \in \Sigma : \alpha(H_{\phi}) = 0\}$ is called the \emph{flag type} of the flow $\phi$. In particular, $H_{\phi} \in \fa_{\Theta(\phi)}$.%
\item[(ii)] The induced flow on $\E_{\Theta}$ admits only one attractor component $\MC^+_{\Theta} = \MC_{\Theta}(1)$ and one repeller component $\MC^-_{\Theta} = \MC_{\Theta}(w_0)$. Moreover, the attractor component $\MC^+_{\Theta(\phi)}$ is given as the image of a continuous section $\sigma_{\phi}:X\rightarrow\E_{\Theta(\phi)}$, i.e.,%
\begin{equation*}
  \left(\MC^+_{\Theta(\phi)}\right)_x = \sigma_{\phi}(x)\quad \mbox{for all } x\in X.%
\end{equation*}
\end{enumerate}
\end{theorem}

For the Morse sets on the maximal flag bundle $\E$ we also write $\MC(w)$. Alternatively, the Morse sets can be described in terms of a block reduction of $\phi$ as follows (cf.~\cite[Prop.~5.4]{SMS}).%

\begin{proposition}
The set $Q_{\phi} = h_{\phi}^{-1}(H_{\phi})$ is a $\phi$-invariant subbundle of $Q$ with structural group $Z_{\phi} := Z_{H_{\phi}}$, called a block reduction of $\phi$. There exists a $K_{\phi}$-reduction $R_{\phi} \subset Q_{\phi}$ ($K_{\phi} = Z_{\phi} \cap K$), i.e., a subbundle with structural group $K_{\phi}$. Then%
\begin{equation}\label{eq_morsecomponents_seconddescr}
  \MC_{\Theta}(w) = \left\{ q \cdot wb_{\Theta}\ :\ q\in Q_{\phi} \right\} = \left\{ r\cdot wb_{\Theta}\ :\ r \in R_{\phi} \right\}.%
\end{equation}
\end{proposition}

\subsection{Iwasawa Decomposition and $\fa$-Cocycle}\label{subsec_acocycle}%

Consider a $G$-principal bundle $\pi:Q\rightarrow X$ as above and let $G = KAN^+$ be an Iwasawa decomposition of $G$. Then there exists a $K$-reduction $R\subset Q$, i.e., a subbundle with structural group $K$. An Iwasawa decomposition of $Q$ is then given by $Q = R \cdot AN^+$, and we can write each $q\in Q$ in a unique way as $q = r \cdot hn$ with $r\in R$, $h\in A$ and $n\in N^+$. We denote by%
\begin{equation*}
  \rmR:Q \rightarrow R,\qquad \rmA:Q \rightarrow A,%
\end{equation*}
the corresponding (continuous) projections from $Q$ to $R$ and from $Q$ to $A$, respectively. The exponential map of $G$ maps $\fa$ bijectively onto $A$. Writing $\log$ for the inverse of $\exp|_{\fa}$, we define%
\begin{equation*}
  \rma(q) := \log \rmA(q),\quad \rma:Q\rightarrow \fa.%
\end{equation*}
Let $\phi_t:Q\rightarrow Q$ be again a flow of automorphisms. Then also%
\begin{equation*}
  \phi_t^R:R \rightarrow R,\quad \phi_t^R(r) = \rmR(\phi_t(r)),%
\end{equation*}
is a flow, and a continuous additive cocycle over $\phi_t^R$ is given by%
\begin{equation*}
  \rma^{\phi}:\R\tm R \rightarrow \fa,\quad \rma^{\phi}(t,r) := \rma(\phi_t(r)).%
\end{equation*}
In the following, by abuse of notation, we only write $\rma$ for $\rma^{\phi}$. Then $\rma$ induces a cocycle over the flow on the maximal flag bundle $\E = Q \tm_G \F = R \tm_K \F$ by%
\begin{equation*}
  \rma(t,\xi) := \rma(t,r),\quad \xi = r \cdot b_0.%
\end{equation*}
We call this cocycle the $\fa$-cocycle over $\phi_t$. For the time-reversed flow we have the additive cocycles%
\begin{equation*}
  \rma^*(t,r) := \log\rmA(\phi_{-t}(r)),\qquad \rma^*(t,r\cdot b_0) := \rma^*(t,r),%
\end{equation*}
and it holds that (cf.~\cite[Prop.~14]{ASM})%
\begin{equation}\label{eq_timeinv_acoc}
  \rma^*(t,\xi) = -\rma(t,\phi_{-t}(\xi)).%
\end{equation}

The following lemma can be found in \cite[Lem.~6.1]{ASM}.%

\begin{lemma}\label{lem_alvessanmartin}
Let $\Theta\subset\Sigma$. If $\beta:\fa\rightarrow V$ is a linear map into an $\R$-vector space $V$, which annihilates on $\fa(\Theta)$, then the cocycle $\rma_{\beta} := \beta \circ \rma$ satisfies%
\begin{equation*}
  \rma_{\beta}(t,r) = \rma_{\beta}(t,r\cdot k),\quad \forall k\in K_{\Theta}.%
\end{equation*}
\end{lemma}

With the aid of this lemma one can show that over the Morse sets there exists a factorization of the $\fa$-cocycle. Let $\pi_{\Theta}:\E\rightarrow\E_{\Theta}$ be the canonical projection given by $r\cdot b_0 \mapsto r\cdot b_{\Theta}$. Since by \eqref{eq_morsecomponents_seconddescr} the Morse sets of the induced flow are given by $\MC_{\Theta}(w) = R_{\phi} \cdot wb_{\Theta}$, it holds that $\pi_{\Theta}(\MC(w)) = \MC_{\Theta}(w)$.%

Now consider the subsets of the roots given by%
\begin{equation*}
  \Pi^{\pm}_{\phi,\Theta,w} := \left\{\alpha\in\Pi^{\pm}\backslash\langle\Theta(\phi)\rangle\ :\ w^{-1}\alpha\in\Pi^-\backslash\langle\Theta\rangle\right\} = \Pi^{\pm}\backslash\langle\Theta(\phi)\rangle \cap w(\Pi^-\backslash\langle\Theta\rangle),%
\end{equation*}
and define%
\begin{equation*}
  \sigma^{\pm}_{\Theta,w} := \sum_{\alpha\in\Pi^{\pm}_{\phi,\Theta,w}} n_{\alpha}\alpha,\qquad n_{\alpha} := \dim\fg_{\alpha}.%
\end{equation*}

\begin{lemma}\label{lem_sigmapmann}
The functionals $\sigma^{\pm}_{\Theta,w}$ annihilate on $\fa(\Theta(\phi) \cap w\langle\Theta\rangle)$.%
\end{lemma}

\begin{proof}
We only give the proof for $\sigma^+_{\Theta,w}$, since for $\sigma^-_{\Theta,w}$ it works analogously. Denote by $(\cdot,\cdot)$ the inner product on $\fa^*$ dual to $B_{\theta}$, i.e.,%
\begin{equation*}
  (\alpha,\beta) = B_{\theta}(H_{\alpha},H_{\beta}).%
\end{equation*}
Consider $\beta \in \Theta(\phi) \cap w\langle\Theta\rangle$ and the associated $\beta$-reflection (the orthogonal reflection at $\beta^{\bot}$), i.e., the map%
\begin{equation}\label{eq_betarefl}
  r_{\beta}(\alpha) = \alpha - 2\frac{(\alpha,\beta)}{(\beta,\beta)}\beta.%
\end{equation}
Then $r_{\beta}(\Pi^+_{\phi,\Theta,w}) = \Pi^+_{\phi,\Theta,w}$. Indeed, let $H\in\cl\fa^+$ such that $\Theta = \Theta(H)$. Then for every $\alpha\in\Pi^+_{\phi,\Theta,w}$ it holds that%
\begin{equation*}
  r_{\beta}(\alpha)(H_{\phi}) = \alpha(H_{\phi}) - 2\frac{(\alpha,\beta)}{(\beta,\beta)}\beta(H_{\phi}) > 0,%
\end{equation*}
since $\alpha(H_{\phi})>0$ and $\beta(H_{\phi}) = 0$. Moreover,%
\begin{equation*}
  w^{-1}r_{\beta}(\alpha)(H) = w^{-1}\alpha(H) - 2\frac{(\alpha,\beta)}{(\beta,\beta)}w^{-1}\beta(H) < 0.%
\end{equation*}
Indeed, since $w^{-1}\alpha \in \Pi^-$ and $w^{-1}\alpha \notin \langle\Theta(H)\rangle$, we have  $w^{-1}\alpha(H) < 0$, and since $w^{-1}\beta \in \langle\Theta\rangle = \langle\Theta(H)\rangle$, we have $w^{-1}\beta(H)=0$. Hence, we obtain $r_{\beta}(\Pi^+_{\phi,\Theta,w}) \subset \Pi^+_{\phi,\Theta,w}$ and since $r_{\beta}$ is injective, equality holds. An elementary computation shows that for every $w\in\WC$ it holds that $w\fg_{\alpha} = \fg_{w\alpha}$ and consequently $n_{\alpha} = n_{w\alpha}$, which yields%
\begin{equation*}
  r_{\beta}(\sigma^+_{\Theta,w}) = \sum_{\alpha\in\Pi^+_{\phi,\Theta,w}}n_{r_{\beta}(\alpha)}r_{\beta}(\alpha) = \sum_{\gamma\in r_{\beta}(\Pi^+_{\phi,\Theta,w})}n_{\gamma}\gamma = \sigma^+_{\Theta,w}.%
\end{equation*}
Looking at the definition \eqref{eq_betarefl} of $r_{\beta}$, one sees that this is equivalent to $\sigma^+_{\Theta,w}(H_{\beta}) = (\sigma^+_{\Theta,w},\beta) = 0$. Since $\beta$ was chosen arbitrarily in $\Theta(\phi)\cap w\langle\Theta\rangle$, the assertion follows.%
\end{proof}

In the rest of this section, we assume that the $K$-reduction $R$ is chosen such that $R_{\phi} \subset R$ (cf.~\cite[Sec.~4.3]{SMS}), and we use the description of the Morse sets on $\E_{\Theta}$ given by \eqref{eq_morsecomponents_seconddescr}.%

\begin{corollary}
For fixed $\Theta\subset\Sigma$ and $w\in\WC$ assume that $\langle\Theta(\phi)\rangle \subset w\langle\Theta\rangle$. Then the maps%
\begin{equation*}
  \rma^{\pm}_{\Theta,w}(t,r\cdot wb_{\Theta}) := \sigma^{\pm}_{\Theta,w}(\rma(t,r)),\quad \rma^{\pm}_{\Theta,w}:\R \tm \MC_{\Theta}(w) \rightarrow \R,%
\end{equation*}
are well-defined additive cocycles over the flow $\phi$ restricted to $\MC_{\Theta}(w)$.%
\end{corollary}

\begin{proof}
We have to show two things:%
\begin{enumerate}
\item[(i)] If $r \cdot wb_{\Theta} = r' \cdot wb_{\Theta}$, then $\sigma^{\pm}_{\Theta,w}(\rma(t,r)) = \sigma^{\pm}_{\Theta,w}(\rma(t,r'))$.%
\item[(ii)] If $w,w' \in \WC$ are such that $\MC_{\Theta}(w) = \MC_{\Theta}(w')$, then%
\begin{equation*}
  \sigma^{\pm}_{\Theta,w}(\rma(t,r)) = \sigma^{\pm}_{\Theta,w'}(\rma(t,r)).%
\end{equation*}
\end{enumerate}
To show (i), assume that $r\cdot wb_{\Theta} = r'\cdot wb_{\Theta}$ with $r,r' \in R_{\phi}$. Then there exists $k\in K_{\phi}$ with $r' = r\cdot k$ (since $K_{\phi}$ is the structural group of $R_{\phi}$ and thus acts transitively on the fibers) and hence $k \cdot wb_{\Theta} = wb_{\Theta}$. This implies $k\in K_{\Theta(\phi)\cap w\Theta}$. Indeed, note that \eqref{eq_kthetaintersec} implies $K_{\Theta(\phi)\cap w\Theta} = K_{\Theta(\phi)} \cap K_{w\Theta}$. However, $K_{\phi} = K_{\Theta(\phi)}$, so $k \in K_{\Theta(\phi)}$. To show that also $k\in K_{w\Theta}$, we use the Langlands decomposition \eqref{eq_langlands}. Then $k \cdot wb_{\Theta} = wb_{\Theta}$ translates into $w^{-1}kw K_{\Theta}AN^+ = K_{\Theta}AN^+$. By uniqueness, this implies $w^{-1}kw \in K_{\Theta}$, i.e., $\Ad(k)wH = wH$ for all $H\in\fa_{\Theta}$, or equivalently $\Ad(k)H = H$ for all $H \in w\fa_{\Theta} = \fa_{w\Theta}$, so $k \in K_{w\Theta}$. Lemma \ref{lem_alvessanmartin} together with Lemma \ref{lem_sigmapmann} then gives $\sigma^{\pm}_{\Theta,w}(\rma(t,r)) = \sigma^{\pm}_{\Theta,w}(\rma(t,r'))$.%
 
Now we show (ii). The Morse sets are parametrized by the double coset space $\WC_{\Theta(\phi)}\backslash\WC/\WC_{\Theta}$. Hence, if $\MC_{\Theta}(w)=\MC_{\Theta}(w')$, then $w' = w_1ww_2$ with $w_1\in\WC_{\Theta(\phi)}$ and $w_2 \in \WC_{\Theta}$. Using that $n_{\alpha} = n_{w\alpha}$ for each $w\in\WC$ (cf.~proof of Lemma \ref{lem_sigmapmann}), we get%
\begin{eqnarray*}
  \sigma^{\pm}_{\Theta,w} \circ w_1^{-1} &=& \sum_{\alpha\in\Pi^{\pm}_{\phi,\Theta,w}}n_{\alpha}(\alpha \circ w_1^{-1})\\
                                         &=& \sum_{\alpha\in\Pi^{\pm}_{\phi,\Theta,w}}n_{w_1\alpha}w_1\alpha = \sum_{\beta \in w_1\Pi^{\pm}_{\phi,\Theta,w}}n_{\beta}\beta.%
\end{eqnarray*}
Using that $\Pi^+ \backslash \langle\Theta(\phi)\rangle = \{\alpha\in\Pi : \alpha(H_{\phi})>0\}$ and $H_{\phi}\in\fa_{\Theta(\phi)}$ (and the analogous statement for $\Pi^-$ and $\Theta$) we find that%
\begin{eqnarray*}
  w_1\Pi^{\pm}_{\phi,\Theta,w} &=& w_1\left(\Pi^{\pm}\backslash\langle\Theta(\phi)\rangle\right) \cap w_1w\left(\Pi^-\backslash\langle\Theta\rangle\right)\\
                               &=& \left(\Pi^{\pm} \backslash \langle\Theta(\phi)\rangle\right) \cap w'w_2^{-1}\left(\Pi^-\backslash\langle\Theta\rangle\right)\\
                               &=& \left(\Pi^{\pm} \backslash \langle\Theta(\phi)\rangle\right) \cap w'\left(\Pi^-\backslash\langle\Theta\rangle\right) = \Pi^{\pm}_{\phi,\Theta,w'}.%
\end{eqnarray*}
Hence, $\sigma^{\pm}_{\Theta,w} \circ w_1^{-1} = \sigma^{\pm}_{\Theta,w'}$. If we write%
\begin{equation*}
  \rma(t,r) = \rma_1(t,r) + \rma_2(t,r) \in \fa_{\Theta(\phi)} \oplus \fa(\Theta(\phi)),%
\end{equation*}
then, since $w_1$ acts trivially on $\fa_{\Theta(\phi)}$,%
\begin{equation*}
  w_1^{-1}\rma(t,r) = \rma_1(t,r) + w_1^{-1}\rma_2(t,r).%
\end{equation*}
By assumption $\langle\Theta(\phi)\rangle \subset w\langle\Theta\rangle$ and consequently $\Theta(\phi) \cap w\langle\Theta\rangle = \Theta(\phi)$, which by Lemma \ref{lem_sigmapmann} implies that $\sigma^{\pm}_{\Theta,w}$ vanishes on $\fa(\Theta(\phi))$. Therefore, we obtain%
\begin{eqnarray*}
   \sigma^{\pm}_{\Theta,w'}(\rma(t,r)) &=& \sigma^{\pm}_{\Theta,w}(w_1^{-1}\rma(t,r))\\
          &=& \sigma^{\pm}_{\Theta,w}(\rma_1(t,r)) + \underbrace{\sigma^{\pm}_{\Theta,w}(w_1^{-1}\rma_2(t,r))}_{=0}\\
          &=& \sigma^{\pm}_{\Theta,w}(\rma_1(t,r) + \rma_2(t,r)) = \sigma^{\pm}_{\Theta,w}(\rma(t,r)),%
\end{eqnarray*}
concluding the proof.%
\end{proof}

The tangent space at the base point $b_{\Theta(\phi)} \in \F_{\Theta(\phi)}$ can be identified with the nilpotent Lie algebra%
\begin{equation*}
  \fn^-_{\phi} := \fn^-_{\Theta(\phi)} = \sum_{\alpha\in\Pi^-\backslash\langle\Theta(\phi)\rangle}\fg_{\alpha},%
\end{equation*}
since this is a complement of $\fp_{\Theta(\phi)}$ in $\fg$. The group $Z_{\phi}$ normalizes $\fn^-_{\phi}$ and therefore acts on $\fn^-_{\phi}$ via the adjoint action. Then we can consider the bundle%
\begin{equation*}
  \VC_{\phi} = Q_{\phi} \tm_{Z_{\phi}} \fn^-_{\phi} \rightarrow X.%
\end{equation*}
Because the $Z_{\phi}$-action is linear, this is a vector bundle and the flow $\Phi_t$, induced by $\phi_t$ on $\VC_{\phi}$, is linear. Now we define%
\begin{equation*}
  \B_{\phi} := Q_{\phi} \cdot N_{\phi}^- b_{\Theta(\phi)}%
\end{equation*}
and the map $\Psi:\VC_{\phi}\rightarrow\B_{\phi}$,%
\begin{equation}\label{eq_psidef}
  \Psi(q \cdot X) := q \cdot (\exp X)b_{\Theta(\phi)},\quad q \in Q_{\phi},\ X \in \fn^-_{\phi}.%
\end{equation}
The following proposition can be found in \cite[Prop.~5.5]{SMS}.%

\begin{proposition}
The following statements hold:%
\begin{enumerate}
\item[(i)] $\B_{\phi}$ is an open and dense $\phi$-invariant subset of $\E_{\Theta(\phi)}$ which contains the attractor component $\MC^+_{\Theta(\phi)} = \Psi(\VC^0_{\phi})$, where $\VC^0_{\phi}$ is the zero section of $\VC_{\phi}$.%
\item[(ii)] $\Psi$ is a homeomorphism which conjugates $\phi$ and $\Phi$, i.e.,%
\begin{equation}\label{eq_conjid}
  \phi_t(\Psi(v)) = \Psi(\Phi_t(v)),\quad v \in \VC_{\phi}.%
\end{equation}
\end{enumerate}
\end{proposition}

We can endow $\VC_{\phi}\rightarrow X$ with the metric%
\begin{equation*}
  (r \cdot X,r \cdot Y) := B_{\theta}(X,Y),\quad r\in R_{\phi},\ X,Y\in\fn^-_{\phi}.%
\end{equation*}
Indeed, this defines a metric on all of $\VC_{\phi}$, which follows from the Iwasawa decomposition $Q_{\phi} = R_{\phi} \cdot AN^+(\phi)$, where $\fn^-_{\phi}$ is normalized by $AN^+(\phi)$, $N^+(\phi) = N^+(\Theta(\phi))$.%

\begin{proposition}\label{prop_cocycle_linearest}
There exist constants $\mu,B\in\R$ with $\mu>0$ such that%
\begin{equation*}
  \alpha(\rma(\tau,\xi)) \geq \mu\tau + B,\quad \forall\tau\geq0,%
\end{equation*}
for all $\xi\in\MC^+ = \MC(1)$ (the attractor component in $\E$) and $\alpha\in\Pi^+\backslash\langle\Theta(\phi)\rangle$.%
\end{proposition}

\begin{proof}
Write $\xi\in\MC^+$ as $\xi = r\cdot b_0$, $r\in R_{\phi}$. Then, using \eqref{eq_ztheta_iwasawa}, for every $\tau\geq0$ we have the Iwasawa decomposition%
\begin{equation*}
  \phi_{\tau}(r) = r_{\tau} \cdot a_{\tau}n_{\tau} \in R_{\phi} \cdot AN^+(\phi),%
\end{equation*}
and it holds that%
\begin{equation*}
  \rma(\tau,\xi) = \log a_{\tau}.%
\end{equation*}
On the other hand, if $v = r\cdot Y\in\VC_{\phi}$ with $Y\in\fn^-_{\phi}$, then $|v| = |Y|_{\theta}$ by definition of the metric on $\VC_{\phi}$. From the conjugacy identity \eqref{eq_conjid} and \eqref{eq_psidef} it follows that%
\begin{equation*}
  \Phi_{\tau}(v) = \Phi_{\tau}(r \cdot Y) = \Psi^{-1}(\phi_{\tau}(r)(\exp Y)b_{\Theta(\phi)}) = \phi_{\tau}(r) \cdot Y.%
\end{equation*}
Then we find that%
\begin{equation*}
  |\Phi_{\tau}(v)| = |\phi_{\tau}(r)\cdot Y| = |r_{\tau} \cdot \Ad(a_{\tau}n_{\tau})Y| = |\Ad(a_{\tau}n_{\tau})Y|_{\theta}.%
\end{equation*}
Since $N^+(\phi)$ centralizes $\fn^-_{\phi}$, we have $\Ad(n_{\tau})Y=Y$ so that $|\Phi_{\tau}(v)| = |\Ad(a_{\tau})Y|_{\theta}$. Therefore,%
\begin{equation*}
  \|\Phi_{\tau}\|_x = \left\|\Ad(a_{\tau})|_{\fn^-_{\phi}}\right\|_{\theta},\quad x = \pi(\xi) \in X.%
\end{equation*}
The operator $\Ad(a_{\tau})|_{\fn^-_{\phi}}$ is positive definite (one can write $\Ad(a_{\tau}) = \rme^{\ad(\log a_{\tau})}$), so that $\|\Ad(a_{\tau})|_{\fn^-_{\phi}}\|_{\theta}$ is equal to the greatest eigenvalue. The eigenvalues are given by $\rme^{\alpha(\log a_{\tau})}$, $\alpha\in\Pi^-\backslash\langle\Theta(\phi)\rangle$, since the eigenvalues of $\ad(\log a_{\tau})|_{\fn^-_{\phi}}$ are $\alpha(\log a_{\tau})$, $\alpha(H_{\phi})<0$. This implies%
\begin{equation*}
  \log\|\Phi_{\tau}\|_x \geq -\alpha(\log a_{\tau}) = -\alpha(\rma(\tau,\xi)),\quad \forall \alpha\in\Pi^+\backslash\langle\Theta(\phi)\rangle.%
\end{equation*}
For every linear flow, whose zero section is an attractor, there are $C,\mu>0$ (cf.~\cite[Lem.~5.2.7]{CKl}) with%
\begin{equation*}
  \|\Phi_{\tau}\|_x \leq C\rme^{-\mu\tau},\quad \tau\geq0,\ x\in X.%
\end{equation*}
With $B = \log C^{-1}$ we obtain the assertion.%
\end{proof}

\section{Right-Invariant Control Systems and Hyperbolic Chain Control Sets}\label{sec_hyperbolic}%

In this section, we study control-affine systems on flag manifolds induced by right-invariant systems on the associated semisimple Lie group. In particular, we provide a characterization of the chain control sets which are uniformly hyperbolic in terms of the flag type of the control flow.%

Consider a connected semisimple non-compact Lie group $G$ with finite center, and choose an Iwasawa decomposition $G=KAN^{\pm}$.  A control-affine system%
\begin{equation*}
  \dot{g}(t) = f_0(g(t)) + \sum_{i=1}^mu_i(t)f_i(g(t)),\quad u\in\UC,%
\end{equation*}
where $f_0,f_1,\ldots,f_m$ are right-invariant vector fields on $G$, is called a \emph{right-invariant control system}. The set $\UC$ of control functions is given by%
\begin{equation*}
  \UC = \left\{u:\R\rightarrow\R^m\ :\ u \mbox{ measurable with } u(t) \in U \mbox{ a.e.}\right\},%
\end{equation*}
where $U\subset\R^m$ is a compact convex set with $0\in\inner U$.
It can easily be shown that each trajectory of this system is defined on the whole time axis (cf.~\cite[Lem.~2.1]{JSu}). Then $\UC$, endowed with the weak$^*$-topology of $L^{\infty}(\R,\R^m) = L^1(\R,\R^m)^*$, is a compact metric space, and the associated control flow%
\begin{equation*}
  \phi_t:\UC\tm G \rightarrow \UC\tm G,\quad \phi_t(u,g) = (\theta_tu,\varphi(t,g,u)),%
\end{equation*}
is a continuous dynamical system. Moreover, $\phi$ is a flow of automorphisms of the trivial $G$-principal bundle $\pi:\UC\tm G\rightarrow\UC$ with chain transitive base flow $\theta$.%

Such a system induces on each of the flag manifolds $\F_{\Theta}$ a control-affine system%
\begin{equation*}
  \dot{x}_{\Theta}(t) = \overline{f}_0(x_{\Theta}(t)) + \sum_{i=1}^mu_i(t)\overline{f}_i(x_{\Theta}(t)),\quad u\in\UC,%
\end{equation*}
with vector fields%
\begin{equation*}
  \overline{f}_i(\pi_{\Theta}(g)) = \left(\rmd\pi_{\Theta}\right)_g f_i(g),%
\end{equation*}
where here $\pi_{\Theta}:G\rightarrow\F_{\Theta}$ denotes the canonical projection $g \mapsto gP_{\Theta}$. We have%
\begin{equation*}
  \varphi_{\Theta}(t,x,u) \equiv \psi_{t,u} \cdot x,\quad \psi_{t,u} := \varphi_{t,u}(1),%
\end{equation*}
where $1\in G$ denotes the neutral element and $\cdot$ the canonical left action of $G$ on $\F_{\Theta}$, $h \cdot gP_{\Theta} = (hg)P_{\Theta}$.%

By Theorem \ref{thm_morsesets} there exist $H_{\phi}\in\cl\fa^+$ and a continuous map%
\begin{equation*}
  h_{\phi}:\UC \tm G \rightarrow \Ad(G)H_{\phi}%
\end{equation*}
with the properties%
\begin{enumerate}
\item[(i)] $h_{\phi}(\phi_t(u,g)) = h_{\phi}(u,g)$, $u\in\UC$, $g\in G$, $t\in\R$,%
\item[(ii)] $h_{\phi}(u,hg) = \Ad(g^{-1})h_{\phi}(u,h)$, $u\in\UC$, $h,g\in G$.%
\end{enumerate}
In the following, we denote the control flow on $\F_{\Theta}$ also by $\phi$ and the transition map by $\varphi$. According to Theorem \ref{thm_morsesets}, the Morse sets of $\phi$ on $\F_{\Theta}$ can be described fiberwise as%
\begin{equation}\label{eq_mcthetafiberwise}
  \MC_{\Theta}(w)_u = \{u\} \tm g \cdot \fix_{\Theta}(h_{\phi}(u,g),w),\quad (u,g)\in\UC \tm G.%
\end{equation}
The continuous map defined by%
\begin{equation*}
  \rmh:\UC \rightarrow \Ad(G)H_{\phi},\quad \rmh(u) := h_{\phi}(u,1),%
\end{equation*}
satisfies the following properties.%

\begin{lemma}\label{lem_rmhprops}\
\begin{enumerate}
\item[(i)] $\rmh(u) = \Ad(g)H_{\phi}$ for all $(u,g) \in Q_{\phi}$.%
\item[(ii)] $\rmh(\theta_tu) = \Ad(\psi_{t,u})\rmh(u)$ for all $u\in\UC$ and $t\in\R$.%
\end{enumerate}
\end{lemma}

\begin{proof}
Item (i) follows from $G$-equivariance:%
\begin{equation*}
  \rmh(u) = h_{\phi}(u,1) = h_{\phi}(u,gg^{-1}) = \Ad(g)h_{\phi}(u,g) = \Ad(g)H_{\phi}.%
\end{equation*}
Item (ii) follows from $G$-equivariance and $\phi$-invariance:%
\begin{eqnarray*}
  \rmh(\theta_tu) &=& h_{\phi}(\theta_tu,1) = h_{\phi}(\theta_tu,\psi_{t,u}\psi_{t,u}^{-1})\\
                  &=& \Ad(\psi_{t,u})h_{\phi}(\theta_tu,\psi_{t,u}) = \Ad(\psi_{t,u})h_{\phi}(u,1) = \Ad(\psi_{t,u})\rmh(u).%
\end{eqnarray*}
\end{proof}

Now from \eqref{eq_mcthetafiberwise}, putting $g=1$, it follows that%
\begin{equation*}
  \MC_{\Theta}(w)_u = \{u\} \tm \fix_{\Theta}(\rmh(u),w).%
\end{equation*}
The chain control sets are the projections of the Morse sets to $\F_{\Theta}$ (see Proposition \ref{prop_ccs}) and thus are given by%
\begin{equation}\label{eq_chaincontrolset}
  E_{\Theta,w} := \bigcup_{u\in\UC}\fix_{\Theta}(\rmh(u),w).%
\end{equation}
In order to study the linearization of the control flow on the chain control sets, we introduce some additional notation. If $X\in\fg$, we also write $X$ for the vector field on $\F_{\Theta}$ given by%
\begin{equation*}
  X(x) = \frac{\rmd}{\rmd t}\Bigl|_{t=0}\rme^{tX} \cdot x.%
\end{equation*}
For every linear subspace $\fl \subset \fg$ we put%
\begin{equation*}
  \fl \cdot x := \left\{X(x) \in T_x\F_{\Theta}\ :\ X\in\fl\right\}.%
\end{equation*}
Each $g\in G$ acts as a diffeomorphism on $\F_{\Theta}$ and it follows from an elementary computation that%
\begin{equation}\label{eq_diffeoder}
  (\rmd g)_xX(x) = (\Ad(g)X)(gx).%
\end{equation}

Take $(u,x) \in \MC_{\Theta}(w)$, i.e., $x\in\fix_{\Theta}(\rmh(u),w)$. By \eqref{eq_morsecomponents_seconddescr} it holds that%
\begin{equation*}
  \MC_{\Theta}(w) = R_{\phi}\cdot wb_{\Theta},\quad R_{\phi} = \left\{(u,k)\in\UC\tm K\ :\ \rmh(u) = \Ad(k)H_{\phi}\right\},%
\end{equation*}
and therefore $x = k_u \cdot wb_{\Theta}$ with $\rmh(u) = \Ad(k_u)H_{\phi}$. Then we can describe the tangent space to $\F_{\Theta}$ at $x$ by%
\begin{equation}\label{eq_fthetatangentspace}
  T_x\F_{\Theta} = \Ad(k_u)w\fn^-_{\Theta} \cdot x,%
\end{equation}
since $\fn^-_{\Theta}$ is a complement of $\fp_{\Theta}$ in $\fg$. We define%
\begin{eqnarray*}
  \SC_{\Theta,w}(u,x) := \fn^-_{\rmh(u)} \cdot x, &\quad& \UC_{\Theta,w}(u,x) := \fn^+_{\rmh(u)}\cdot x,\\
  \CC_{\Theta,w}(u,x) \!\!\!&:=&\!\!\! \fz_{\rmh(u)} \cdot x.%
\end{eqnarray*}
Note that, in general, $\rmh(u)$ is not contained in the initial maximal abelian subspace $\fa$, but in the one conjugated by $k_u$, and%
\begin{equation}\label{eq_fnpm_u}
  \fn^{\pm}_{\rmh(u)} = \fn^{\pm}_{\Ad(k_u)H_{\phi}} = \Ad(k_u)\fn^{\pm}_{\phi}.%
\end{equation}

\begin{proposition}\label{prop_centerbundle}
It holds that%
\begin{equation*}
  \CC_{\Theta,w}(u,x) = T_x\fix_{\Theta}(\rmh(u),w).%
\end{equation*}
\end{proposition}

\begin{proof}
Let $\gamma:\R\rightarrow\fix_{\Theta}(\rmh(u),w)$ be a smooth curve with $\gamma(0) = x = k \cdot wb_{\Theta}$. Since $\fix_{\Theta}(\rmh(u),w) = k \cdot \fix_{\Theta}(H_{\phi},w) = (kZ_{H_{\phi}}) \cdot wb_{\Theta}$, we can write $\gamma(t) = k\widetilde{\gamma}(t) \cdot wb_{\Theta}$ with $\widetilde{\gamma}(t) \in Z_{H_{\phi}}$ and $\widetilde{\gamma}(0) = 1$. Indeed, we can assume that $\widetilde{\gamma}(t) = \rme^{tX}$ with $X \in \fz_{H_{\phi}}$. Then, writing $i_k(g)=kgk^{-1}$, we find%
\begin{eqnarray*}
  \frac{\rmd}{\rmd t}\Bigl|_{t=0}\gamma(t) &=& \frac{\rmd}{\rmd t}\Bigl|_{t=0} k\rme^{tX} \cdot wb_{\Theta} = \frac{\rmd}{\rmd t}\Bigl|_{t=0} i_k(\rme^{tX}) \cdot x\\
                                           &=& \frac{\rmd}{\rmd t}\Bigl|_{t=0} \rme^{t\Ad(k)X} \cdot x = (\Ad(k)X)(x) \in \fz_{\rmh(u)} \cdot x,%
\end{eqnarray*}
since $\ad(\Ad(k)X)\rmh(u) = [\Ad(k)X,\rmh(u)] = [\Ad(k)X,\Ad(k)H_{\phi}] = \Ad(k)[X,H_{\phi}] = \Ad(k)\ad(X)H_{\phi} = 0$. Hence, $T_x\fix_{\Theta}(\rmh(u),w) \subset \fz_{\rmh(u)}\cdot x$. The equality follows, since an arbitrary choice of $X\in\fz_{H_{\phi}}$ yields a corresponding curve in $\fix_{\Theta}(\rmh(u),w)$.%
\end{proof}

\begin{proposition}\label{prop_dec}
For every $(u,x)\in\MC_{\Theta}(w)$ we have the decomposition%
\begin{equation}\label{eq_tsdec}
  T_x\F_{\Theta} = \SC_{\Theta,w}(u,x) \oplus \CC_{\Theta,w}(u,x) \oplus \UC_{\Theta,w}(u,x).%
\end{equation}
These subspaces are $\phi$-invariant, i.e.,%
\begin{equation*}
  (\rmd\varphi_{t,u})_x\SC_{\Theta,w}(u,x) = \SC_{\Theta,w}(\phi_t(u,x)),%
\end{equation*}
and analogously for the other two subspaces. Moreover, their dimensions are constant on $\MC_{\Theta}(w)$ and they are the fibers of continuous subbundles.%
\end{proposition}

\begin{proof}
We subdivide the proof into four steps.%

\emph{Step 1.} We show that the decomposition \eqref{eq_tsdec} holds. First observe that for every $H\in\fa$ the Lie algebra $\fg$ decomposes as%
\begin{equation*}
  \fg = \fn^-_H \oplus \fz_H \oplus \fn^+_H,%
\end{equation*}
where $\fz_H$ is the centralizer of $H$, which is the $0$-eigenspace of $\ad(H)$, while $\fn^{\pm}_H$ are the sums of the eigenspaces corresponding to positive and negative eigenvalues, respectively. Then%
\begin{equation*}
  T_{b_{\Theta}}\F_{\Theta} = (\rmd\pi_{\Theta})_1 \fn^-_H + (\rmd\pi_{\Theta})_1 \fz_H + (\rmd\pi_{\Theta})_1 \fn^+_H.%
\end{equation*}
Since%
\begin{equation*}
  X(b_{\Theta}) = \frac{\rmd}{\rmd t}\Bigl|_{t=0}\rme^{tX}\cdot b_{\Theta} = \frac{\rmd}{\rmd t}\Bigl|_{t=0}\pi_{\Theta}(\rme^{tX}) = (\rmd\pi_{\Theta})_1 X,%
\end{equation*}
the identity $(\rmd\pi_{\Theta})_1\fn^{\pm}_H = \fn^{\pm}_H\cdot b_{\Theta}$ holds. The analogous argument applies to $\fz_H$. Hence, we have%
\begin{equation*}
  T_{b_{\Theta}}\F_{\Theta} = \fn^-_H\cdot b_{\Theta} + \fz_H \cdot b_{\Theta} + \fn^+_H \cdot b_{\Theta}.%
\end{equation*}
We can write $w = \Ad(k)|_{\fa}$ with $k\in M^*$. Looking at the diffeomorphism $k:\F_{\Theta}\rightarrow\F_{\Theta}$, we obtain%
\begin{eqnarray*}
  T_{wb_{\Theta}}\F_{\Theta} &=& (\rmd k)_{b_{\Theta}}\fn^-_H\cdot b_{\Theta} + (\rmd k)_{b_{\Theta}}\fz_H\cdot b_{\Theta} + (\rmd k)_{b_{\Theta}}\fn^+_H\cdot b_{\Theta}\\
                             &=& \Ad(k)\fn^-_H \cdot wb_{\Theta} + \Ad(k)\fz_H \cdot wb_{\Theta} + \Ad(k)\fn^+_H \cdot wb_{\Theta}\\
                             &=& \fn^-_{\Ad(k)H} \cdot wb_{\Theta} + \fz_{\Ad(k)H} \cdot wb_{\Theta} + \fn^+_{\Ad(k)H} \cdot wb_{\Theta}\\
                             &=& \fn^-_{wH} \cdot wb_{\Theta} + \fz_{wH} \cdot wb_{\Theta} + \fn^+_{wH} \cdot wb_{\Theta}.%
\end{eqnarray*}
Here we used that $\Ad(k)\fg_{\alpha} = \fg_{\Ad(k)\alpha}$. Now let $(u,x)\in\MC_{\Theta}(w)$. Similarly,%
\begin{equation*}
  T_x\F_{\Theta} = \fn^-_{\Ad(k_u)wH}\cdot x + \fz_{\Ad(k_u)wH}\cdot x + \fn^+_{\Ad(k_u)wH}\cdot x,%
\end{equation*}
Putting $H = w^{-1}H_{\phi}$, we arrive at%
\begin{equation}\label{eq_tangentspace_sum}
  T_x\F_{\Theta} = \SC_{\Theta,w}(u,x) + \CC_{\Theta,w}(u,x) + \UC_{\Theta,w}(u,x).%
\end{equation}

\emph{Step 2.} We show that the dimensions of the subspaces are constant and that the decomposition is a direct sum. Consider $\SC_{\Theta,w}(u,x)$. Since $x = k_u \cdot wb_{\Theta}$ with $\rmh(u) = \Ad(k_u)H_{\phi}$, it follows with \eqref{eq_fnpm_u} that%
\begin{equation*}
  \SC_{\Theta,w}(u,x) = \fn^-_{\rmh(u)}\cdot x = \Ad(k_u)\fn^-_{\phi}\cdot x = (\rmd k_u)_{wb_{\Theta}}\left(\fn^-_{\phi}\cdot wb_{\Theta}\right).%
\end{equation*}
This shows that the dimension of $\SC_{\Theta,w}(u,x)$ is constantly equal to that of $\fn^-_{\phi}\cdot wb_{\Theta}$. For the other two subspaces the proof is analogous. To see that \eqref{eq_tangentspace_sum} is a direct sum, we show that the dimensions of the three subspaces exactly sum up to the dimension of the tangent space. By \eqref{eq_fthetatangentspace}, the dimension of the tangent space is $\dim\fn^-_{\Theta}$. The dimensions of $\SC_{\Theta,w}(u,x)$ and $\UC_{\Theta,w}(u,x)$, respectively, are the sums of the dimensions of the $\fg_{\alpha}$ with%
\begin{equation*}
  \alpha \in \Pi^{\pm}_{\phi,\Theta,w} = \Pi^{\pm}\backslash\langle\Theta(\phi)\rangle \cap w(\Pi^-\backslash\langle\Theta\rangle),%
\end{equation*}
since the tangent space at $wb_{\Theta}$ is $w\fn^-_{\Theta} \cdot wb_{\Theta} = (\rmd\widetilde{k})_{b_{\Theta}}\fn^-_{\Theta}\cdot b_{\Theta}$ with $w = \Ad(\widetilde{k})|_{\fa}$. The dimension of $\fz_{\rmh(u)}\cdot x$ is the sum of the dimensions of those $\fg_{\alpha}$ with%
\begin{equation*}
  \alpha \in \langle\Theta(\phi)\rangle \cap w(\Pi^-\backslash \langle\Theta\rangle).%
\end{equation*}
Consequently, the dimensions of the three subspaces sum up to%
\begin{equation*}
  \sum_{\alpha\in w(\Pi^-\backslash\langle\Theta\rangle)}\dim\fg_{\alpha} = \dim\fn^-_{\Theta} = \dim T_x\F_{\Theta}.%
\end{equation*}

\emph{Step 3.} We prove $\phi$-invariance. Using $\varphi_{t,u}(x) = \psi_{t,u} \cdot x$, with Lemma \ref{lem_rmhprops} it follows that%
\begin{eqnarray*}
  (\rmd\varphi_{t,u})_x\SC_{\Theta,w}(u,x) &=& (\rmd\psi_{t,u})_x\fn^-_{\rmh(u)}\cdot x = \left(\Ad(\psi_{t,u})\fn^-_{\rmh(u)}\right)\cdot\varphi_{t,u}(x)\\
  &=& \fn^-_{\Ad(\psi_{t,u})\rmh(u)}\cdot\varphi_{t,u}(x)\\
 &=& \fn^-_{\rmh(\theta_tu)}\cdot \varphi_{t,u}(x) = \SC_{\Theta,w}(\phi_t(u,x)).%
\end{eqnarray*}
For $\UC_{\Theta,w}(u,x)$ and $\CC_{\Theta,w}(u,x)$ the proof works analogously.%

\emph{Step 4.} Finally, we prove continuity of the subbundles. To this end, it suffices to show that the subbundles are closed (cf.~\cite[Lem.~B.1.13]{CKl}). W.l.o.g., consider only $\CC_{\Theta,w}$. Let $(u_n,v_n)\in\CC_{\Theta,w}(u_n,x_n)$ be a sequence converging to some $(u_*,v_*)\in \bigcup_{(u,x)\in\MC_{\Theta}(w)}\{u\}\tm T_x\F_{\Theta}$. We need to show that $v \in \CC_{\Theta,w}(u_*,x_*)$ with $x_* = \mathrm{pr}(v_*)$, $\mathrm{pr}:T\F_{\Theta}\rightarrow \F_{\Theta}$ being the base point projection. Write%
\begin{equation*}
  v_n = (\rmd k_n)_{wb_{\Theta}} X_n(wb_{\Theta}),\quad x_n = k_n \cdot wb_{\Theta},\quad X_n \in \fz_{\phi},\quad \rmh(u_n) = \Ad(k_n)H_{\phi}.%
\end{equation*}
We can decompose $X_n = Y_n \oplus Z_n \in w\fn^-_{\Theta} \oplus w\fp_{\Theta}$, and by $\ad(H_{\phi})$-invariance of $w\fn^-_{\Theta}$ and $w\fp_{\Theta}$ we find that $Y_n,Z_n \in \fz_{\phi}$. Hence, we may assume that $X_n\in w\fn^-_{\Theta}$. Using a $K$-invariant Riemannian metric on $\F_{\Theta}$ with the property \eqref{eq_kinvmetric}, we obtain with $w^{-1} = \Ad(k)|_{\fa}$ that%
\begin{equation*}
  |v_n| = |X_n(wb_{\Theta})| = |(\rmd k)_{wb_{\Theta}} X_n(wb_{\Theta})| = |\Ad(k)X_n(b_{\Theta})| = |X_n|,%
\end{equation*}
and therefore the sequence $(X_n)$ is bounded and we may assume that $X_n \rightarrow X_* \in \fz_{\phi}$. Finally, since $K$ is compact, we may assume that $k_n \rightarrow k_* \in K$. Putting everything together, we end up with%
\begin{equation*}
  v_* = (\rmd k_*)_{wb_{\Theta}} X_*(wb_{\Theta}),\quad x_* = k_* \cdot wb_{\Theta},\quad \rmh(u_*) = \Ad(k_*)H_{\phi},% 
\end{equation*}
where we use continuity of $\rmh$. Hence, $v_* \in \Ad(k_*)\fz_{\phi} \cdot x_* = \CC_{\Theta,w}(u_*,x_*)$, concluding the proof.%
\end{proof}

Now we consider the $\fa$-cocycle on $\UC\tm\F$ over $\phi$, given by%
\begin{equation*}
  \rma:\R\tm\UC\tm\F \rightarrow \fa,\quad (t,u,x) \mapsto \rma(t,u,x) := \log \rmA(\varphi_{t,u}(x)).%
\end{equation*}

\begin{proposition}\label{prop_hyperbolicconstants}
There exist constants $c,\mu>0$ such that for all $(u,x)\in\MC_{\Theta}(w)$ it holds that%
\begin{equation*}
  \|(\rmd\varphi_{t,u})_xv\| \leq c^{-1}\rme^{-\mu t}\|v\| \mbox{\quad for all\ } t\geq0,\ v\in\SC_{\Theta,w}(u,x)%
\end{equation*}
and%
\begin{equation*}
  \|(\rmd\varphi_{t,u})_xv\| \geq c\rme^{\mu t}\|v\| \mbox{\quad for all\ } t\geq0,\ v\in\UC_{\Theta,w}(u,x).%
\end{equation*}
\end{proposition}

\begin{proof}
Let $k\in K$ such that $x = k\cdot wb_{\Theta}$ and $\rmh(u) = \Ad(k)H_{\phi}$, i.e., $(u,k)\in R_{\phi}$. From $\phi$-invariance of $Q_{\phi}$, using \eqref{eq_ztheta_iwasawa}, it follows that%
\begin{equation*}
  \varphi_{t,u}(k) = k_{t,u}a_{t,u}n_{t,u} \in KAN^+(\phi),\quad N^+(\phi) = N^+(\Theta(\phi)).%
\end{equation*}
Consider on $\F_{\Theta}$ a $K$-invariant Riemannian metric with the property \eqref{eq_kinvmetric}. Because of the relation between $k$ und $u$ it holds that%
\begin{equation*}
  \fn^-_{\rmh(u)} \cdot x = (\rmd k)_{wb_{\Theta}}\left(\fn^-_{\phi}\cdot wb_{\Theta}\right).%
\end{equation*}
Now let $v\in\SC_{\Theta,w}(u,x)$ and $\overline{v}\in\fn^-_{\phi}\cdot wb_{\Theta}$ such that $v = (\rmd k)_{wb_{\Theta}}\overline{v}$. Then%
\begin{eqnarray*}
  \left\|(\rmd\varphi_{t,u})_xv\right\| &=& \left\|(\rmd\psi_{t,u})_x v\right\| = \left\|(\rmd\psi_{t,u})_x(\rmd k)_{wb_{\Theta}}\overline{v}\right\|\\
                                        &=& \left\|(\rmd[\varphi_{t,u}(k)])_{wb_{\Theta}}\overline{v}\right\|\\
                                        &=& \left\|(\rmd k_{t,u})_{wb_{\Theta}}(\rmd a_{t,u})_{wb_{\Theta}}(\rmd n_{t,u})_{wb_{\Theta}}\overline{v}\right\|.%
\end{eqnarray*}
Here we use that both $a_{t,u}$ and $n_{t,u}$ fix $wb_{\Theta}$. For $a_{t,u}$ this is proved as follows: We need to show that $\Ad(k^{-1}a_{t,u}k)\fp_{\Theta} = \fp_{\Theta}$, where $w = \Ad(k)|_{\fa}$, $k\in M^*$. Since $\Ad(k)$ normalizes $\fa$, we have $a := k^{-1}a_{t,u}k \in A$ and hence $\Ad(a)\fp_{\Theta} = \fp_{\Theta}$, because $\fa \subset \fp_{\Theta}$. For $n_{t,u}$ the proof is more involved: We can write $w = w_1w^{\phi}$ with $w_1 \in \WC_{\Theta(\phi)}$ and $(w^{\phi})^{-1}(\Theta(\phi)) \subset \Pi^+$ (cf.~\cite[Lem.~1.1.2.15]{War}). We have $x = k \cdot wb_{\Theta} = k' \cdot w^{\phi}b_{\Theta}$ with $k' = kw_1$. From $w_1\in\WC_{\Theta(\phi)}$ it follows that $(u,k')\in R_{\phi}$, since $H_{\phi}\in\fa_{\Theta(\phi)}$ implies%
\begin{equation*}
  h_{\phi}(u,k') = w_1^{-1} \cdot h_{\phi}(u,k) = w_1^{-1} H_{\phi} = H_{\phi}.%
\end{equation*}
Since $(w^{\phi})^{-1}(\Theta(\phi)) \subset \Pi^+$, we have $(w^{\phi})^{-1}N^+(\phi)w^{\phi} \subset N^+$. Now we argue as follows: Instead of $w$ we can choose $w^{\phi}$ as the corresponding representative in the double coset $\WC_{\Theta(\phi)}w\WC_{\Theta}$. So we may assume that $\varphi_{t,u}(k') = k_{t,u}a_{t,u}n_{t,u}$. Then there is $\overline{n}_{t,u}\in N^+$ with%
\begin{equation*}
  n_{t,u} \cdot w^{\phi}b_{\Theta} = w^{\phi} \overline{n}_{t,u}b_{\Theta} = w^{\phi}b_{\Theta},%
\end{equation*}
since $N^+$ stabilizes $b_{\Theta} = \fp_{\Theta}$, which follows from the fact that the Lie subalgebra $\fp_{\Theta}$ contains $\fn^+$ and hence $\Ad(n)\fp_{\Theta} = \fp_{\Theta}$ for every $n\in N^+$.%

Since $N^+(\phi)$ centralizes $\fn^-_{\phi}$, we have that $(\rmd n_{t,u})_{wb_{\Theta}}\overline{v} = \overline{v}$. Together with the fact that the chosen metric is $K$-invariant, we obtain%
\begin{equation*}
  \left\|(\rmd\varphi_{t,u})_xv\right\| = \left\|(\rmd a_{t,u})_{wb_{\Theta}}\overline{v}\right\|.%
\end{equation*}
Because $(\rmd k)_{wb_{\Theta}}$ is an isometry, we conclude that%
\begin{equation*}
  \left\|(\rmd\varphi_{t,u})|_{\SC_{\Theta,w}(u,x)}\right\| = \left\|(\rmd a_{t,u})_{wb_{\Theta}}|_{\fn^-_{\phi}\cdot wb_{\Theta}}\right\| = \left\|\Ad(a_{t,u})|_{\fn^-_{\phi,w}}\right\|,%
\end{equation*}
where%
\begin{equation*}
  \fn^-_{\phi,w} = \sum_{\alpha\in\Pi^-_{\phi,\Theta,w}}\fg_{\alpha},\qquad \Pi^-_{\phi,\Theta,w} = (\Pi^-\backslash\langle\Theta(\phi)\rangle) \cap w(\Pi^-\backslash\langle\Theta\rangle).%
\end{equation*}
The operator $\Ad(a_{t,u})|_{\fn^-_{\phi,w}}$ is positive definite, so that $\|\Ad(a_{t,u})|_{\fn^-_{\phi,w}}\|$ equals the greatest eigenvalue. Since the eigenvalues are given by $\rme^{\alpha(\rma(t,u,k\cdot b_0))}$, $\alpha\in\Pi^-_{\phi,\Theta,w}$, with Proposition \ref{prop_cocycle_linearest} we obtain%
\begin{equation*}
  \left\|(\rmd\varphi_{t,u})_xv\right\| \leq c^{-1}\rme^{-\mu t}\|v\|,%
\end{equation*}
where $c = \rme^B$, since $(u,k\cdot b_0)$ is contained in the attractor component $\MC^+=\MC(1)$ of the flow on $\UC\tm\F$. This proves the first inequality.%

Now let $v\in\UC_{\Theta,w}(u,x)$. It holds that%
\begin{equation*}
  \fn^+_{\rmh(u)} \cdot x = (\rmd k)_{wb_{\Theta}}\left(\fn^+_{\phi} \cdot wb_{\Theta}\right).%
\end{equation*} 
Similarly as above we define the subspace%
\begin{equation*}
  \fn^+_{\phi,w} := \sum_{\alpha\in\Pi^+_{\phi,\Theta,w}}\fg_{\alpha}.%
\end{equation*}
Since for every $t\geq0$ we have%
\begin{equation*}
  \varphi_{-t,u}(k) = k^*_{t,u}a^*_{t,u}n^*_{t,u}%
\end{equation*}
with $a^*_{t,u}n^*_{t,u}\in AN^-(\phi)$, $k^*_{t,u}\in K$, and $\log a^*_{t,u} = \rma^*(t,u,k\cdot b_0)$, we obtain%
\begin{equation*}
  \left\|(\rmd\varphi_{-t,u})|_{\UC_{\Theta,w}(u,x)}\right\| = \left\|(\rmd a^*_{t,u})|_{\fn^+_{\phi} \cdot wb_{\Theta}}\right\| = \left\|\Ad(a^*_{t,u})|_{\fn^+_{\phi,w}}\right\|.%
\end{equation*}
Since $\Ad(a^*_{t,u})|_{\fn^+_{\phi,w}}$ is positive definite, $\|\Ad(a^*_{t,u})|_{\fn^+_{\phi,w}}\|$ equals the greatest eigenvalue. The eigenvalues are given by $\rme^{\alpha(\rma^*(t,u,k\cdot b_0))}$, $\alpha\in\Pi^+_{\phi,\Theta,w}$, and hence Proposition \ref{prop_cocycle_linearest} together with \eqref{eq_timeinv_acoc} gives%
\begin{equation*}
  \left\|(\rmd\varphi_{-t,u})_xv\right\| \leq c^{-1}\rme^{-\mu t}\|v\| \mbox{\quad for all\ } t\geq0,\ (u,x)\in\MC_{\Theta}(w),%
\end{equation*}
and because of the $\phi$-invariance of $\UC_{\Theta,w}$ it follows that $\|(\rmd\varphi_{t,u})_xv\| \geq c\rme^{\mu t}\|v\|$, as claimed.%
\end{proof}

Now we can state the main result of the paper, the characterization of the hyperbolic chain control sets on $\F_{\Theta}$.%

\begin{theorem}
Given $\Theta\subset\Sigma$ and $w\in\WC$, the following are equivalent:%
\begin{enumerate}
\item[(i)] The subbundle $\CC_{\Theta,w}$ is trivial.%
\item[(ii)] The chain control set $E_{\Theta,w}$ is uniformly hyperbolic.%
\item[(iii)] $\langle\Theta(\phi)\rangle \subset w\langle\Theta\rangle$.%
\end{enumerate}
\end{theorem}

\begin{proof}
According to the proof of Proposition \ref{prop_dec}, the dimension of $\CC_{\Theta,w}(u,x)$ is the sum of the dimensions of those $\fg_{\alpha}$ with $\alpha\in\langle\Theta(\phi)\rangle\cap w(\Pi^-\backslash\langle\Theta\rangle)$. It is clear that $\CC_{\Theta,w}$ is trivial iff this intersection is empty. Obviously, $\langle\Theta(\phi)\rangle \subset w\langle\Theta\rangle$ implies that the intersection is empty. Conversely, if it is empty, then%
\begin{equation*}
  \langle\Theta(\phi)\rangle \subset w(\Pi^-\backslash\langle\Theta\rangle)^c = w\Pi^+ \cup w\langle\Theta\rangle,%
\end{equation*}
or equivalently $w^{-1}\langle \Theta(\phi)\rangle \subset \Pi^+ \cup \langle\Theta\rangle$. If $\alpha\in w^{-1}\langle\Theta(\phi)\rangle$, then also $-\alpha \in w^{-1}\langle\Theta(\phi)\rangle$. But $\alpha$ and $-\alpha$ cannot be both in $\Pi^+$, so one of them (and hence both) are in $\langle\Theta\rangle$, which implies that (i) and (iii) are equivalent. From Proposition \ref{prop_hyperbolicconstants} it follows that (i) implies (ii). To finish the proof, it suffices to prove that (ii) implies (i). To this end, assume that (i) does not hold. Then, by Proposition \ref{prop_centerbundle}, the sets $\fix_{\Theta}(\rmh(u),w)$ are submanifolds of dimension at least $1$. Take a constant control function $u$, and let $k_u\in K$ with $\rmh(u)=\Ad(k_u)H_{\phi}$. By Lemma \ref{lem_rmhprops}, we have%
\begin{equation*}
  \fix_{\Theta}(\rmh(u),w) = \fix_{\Theta}(\rmh(\theta_tu),w) =
  \fix_{\Theta}(\Ad(\psi_{t,u})\rmh(u),w) = \varphi_{t,u}(\fix_{\Theta}(\rmh(u),w)).%
\end{equation*}
Hence, $\fix_{\Theta}(\rmh(u),w)$ is invariant under the flow $(\varphi_{t,u})_{t\in\R}$. Then there are two possibilities: (a) all elements of $\fix_{\Theta}(\rmh(u),w)$ are equilibria, in which case the flow is trivial on this set, or (b) there exists at least one nonstationary trajectory. In both cases we have at least one Lyapunov exponent that is zero (in case (b) the Lyapunov exponent in direction of the trajectory vanishes), and hence hyperbolicity cannot hold. %
\end{proof}

\begin{remark}
Note that the hyperbolicity condition is independent of the representative $w\in\WC$. Indeed, if $w' = w_1ww_2$ with $w_1\in\WC_{\Theta(\phi)}$, $w_2\in\WC_{\Theta}$, and $\langle\Theta(\phi)\rangle\subset w\langle\Theta\rangle$, then $\langle\Theta(\phi)\rangle = w_1 \langle\Theta(\phi)\rangle \subset w_1 w\langle\Theta\rangle =  w_1 w\langle w_2\Theta \rangle =  w'\langle\Theta\rangle$.%
\end{remark}

\begin{remark}
One particularly interesting case is the case when $H_{\phi}$ is a regular element, i.e., $H_{\phi} \in \fa^+$. Then $\Theta(\phi) = \emptyset$ and hence all chain control sets on all flag manifolds are hyperbolic.%
\end{remark}

According to a general result in Colonius and Du \cite{CDu}, under appropriate regularity assumptions, a hyperbolic chain control set of a control-affine system is the closure of a control set. We want to give an independent proof for this fact for the hyperbolic chain control sets on $\F_{\Theta}$. If $\langle\Theta(\phi)\rangle \subset w\langle\Theta\rangle$, then the fixed point component $\fix_{\Theta}(\rmh(u),w)$ reduces to a point $x(u)$ for every $u\in\UC$. Hence, we can define the map%
\begin{equation*}
  \sigma_{\Theta,w}:\UC \rightarrow \MC_{\Theta}(w),\quad u \mapsto (u,x(u)),%
\end{equation*}
and we have the following result.%

\begin{proposition}\label{prop_conjugacy}
The map $\sigma_{\Theta,w}$ is a topological conjugacy between the shift flow on $\UC$ and the control flow restricted to $\MC_{\Theta}(w)$.%
\end{proposition}

\begin{proof}
Obviously, the projection from $\MC_{\Theta}(w)$ to $\UC$, i.e., the map $(u,x) \mapsto u$, is an inverse of $\sigma_{\Theta,w}$. Hence, $\sigma_{\Theta,w}$ is bijective. A bijective map between compact Hausdorff spaces is a homeomorphism iff the map or its inverse is continuous. In this case, the inverse is obviously continuous, so $\sigma_{\Theta,w}$ is a homeomorphism. Using Lemma \ref{lem_rmhprops}, the conjugacy identity is proved by%
\begin{eqnarray*}
  \sigma_{\Theta,w}(\theta_tu) &=& (\theta_tu,\fix_{\Theta}(\rmh(\theta_tu),w))\\
  &=& (\theta_tu,\fix_{\Theta}(\Ad(\psi_{t,u})\rmh(u),w))\\
  &=& (\theta_tu,\psi_{t,u}\cdot\fix_{\Theta}(\rmh(u),w)) = \phi_t(\sigma_{\Theta,w}(u)),%
\end{eqnarray*}
which finishes the proof.%
\end{proof}

This easily implies that the chain control set $E_{\Theta,w}$ is the closure of a control set if it has nonempty interior and local accessibility holds.%

\begin{corollary}\label{cor_closure}
Assume that $\langle\Theta(\phi)\rangle \subset w\langle\Theta\rangle$ for some $\Theta\subset\Sigma$ and $w\in\WC$, and that the system on $\F_{\Theta}$ is locally accessible. Then, if $E_{\Theta,w}$ has a nonempty interior, it is the closure of a control set.%
\end{corollary}

\begin{proof}
We use that the shift flow is topologically mixing (see \cite[Prop.~4.1.1]{CKl}). Then Proposition \ref{prop_conjugacy} implies that also the control flow restricted to $\MC_{\Theta}(w)$ is topologically mixing. Intuitively, this means that the dynamics on $\MC_{\Theta}(w)$ is indecomposable, and projecting to $\F_{\Theta}$ it means that approximate controllability holds on the interior of $E_{\Theta,w}$. This is made precise in \cite[Thm.~4.1.3]{CKl}, which immediately implies the assertion.%
\end{proof}

\begin{remark}
A sufficient condition for local accessibility on $\F_{\Theta}$ (for an arbitrary set $\Theta$) is that the semigroup of the right-invariant system on $G$ has nonempty interior, i.e., the positive orbit $S = \OC^+(1)$ of the neutral element. In fact, $\inner S \neq \emptyset$ is equivalent to the Lie algebra rank condition (and hence local accessibility) on $G$, which implies local accessibility on all flag manifolds, since the trajectories on $\F_{\Theta}$ are just the projections of those on $G$, and the projection map is open. The equivalence is seen as follows: If the Lie algebra rank condition on $G$ holds, then Krener's theorem implies that $\inner S \neq \emptyset$. Conversely, if the Lie algebra rank condition does not hold, then the Lie subalgebra $\LC \subset \fg$ generated by the right-invariant vector fields $f_0,f_1,\ldots,f_m$ has dimension smaller than $\dim\fg$. Since we may assume that the vector fields are analytic, $\LC$ is the tangent space to the full orbit $\OC(1)$ at $1$ and hence $\inner \OC(1) = \emptyset$ implying $\inner S = \emptyset$ (since $S\subset\OC(1)$).%
\end{remark}

\begin{remark}
We note that under the assumption that $\inner S \neq \emptyset$, there also exists an algebraic description of the control sets on $\F_{\Theta}$ (see San Martin and Tonelli \cite{STo}). In this case, the control sets $D_{\Theta}(w)$ on $\F_{\Theta}$, just as the chain control sets, are parametrized by the elements of a double coset space $\WC_{\Theta(S)}\backslash\WC/\WC_{\Theta}$, where $\Theta(S)$ is the flag type of the semigroup $S$. Using this interpretation, an alternative proof of Corollary \ref{cor_closure} is possible. The idea is that for any $u\in\UC$ there are periodic functions $u_n\in\UC$ such that $\psi_{\tau_n, u_n}^{M_n}$ are regular elements in $\inner S$ for integers $M_n$ and such that their periodic points $x_n=\fix_{\Theta}(\rmh(u_n),w)$ converge to $x=\fix_{\Theta}(\rmh(u),w)$. Since the set of periodic points of regular elements in $\inner S$ is dense in the control set $D_{\Theta}(w)$ (see \cite{STo}) and $D_{\Theta}(w)\subset E_{\Theta,w}$, we get $E_{\Theta, w}=\cl D_{\Theta}(w)$.%
\end{remark}

For future purposes we describe (in the hyperbolic case) the unstable and stable determinants in terms of the cocycles $\rma^{\pm}_{\Theta,w}$.%

\begin{proposition}
Assume that $\langle\Theta(\phi)\rangle \subset w\langle\Theta\rangle$. Then for all $(u,x)\in\MC_{\Theta}(w)$ and $t\geq0$ it holds that%
\begin{eqnarray*}
  \left|\det(\rmd\varphi_{t,u})|_{\UC_{\Theta,w}(u,x)}\right| &=& \rme^{\rma^+_{\Theta,w}(t,u,x)},\\
  \left|\det(\rmd\varphi_{t,u})|_{\SC_{\Theta,w}(u,x)}\right| &=& \rme^{\rma^-_{\Theta,w}(t,u,x)}.%
\end{eqnarray*}
\end{proposition}

\begin{proof}
Choose $k\in K$ with $x = k \cdot wb_{\Theta}$. Then%
\begin{equation*}
  \UC_{\Theta,w}(u,x) = (\rmd k)_{wb_{\Theta}}\fn^+_{\phi}\cdot wb_{\Theta},%
\end{equation*}
which by the chain rule implies%
\begin{equation*}
  (\rmd\varphi_{t,u})|_{\UC_{\Theta,w}(u,x)} \circ (\rmd k)|_{\fn^+_{\phi}\cdot wb_{\Theta}} = (\rmd[\varphi_{t,u}(k)])|_{\fn^+_{\phi}\cdot wb_{\Theta}},%
\end{equation*}
and therefore%
\begin{equation*}
  \left|\det(\rmd\varphi_{t,u})|_{\UC_{\Theta,w}(u,x)}\right| = \left|(\rmd[\varphi_{t,u}(k)])|_{\fn^+_{\phi}\cdot wb_{\Theta}}\right|,%
\end{equation*}
when we endow $\F_{\Theta}$ with a $K$-invariant Riemannian metric satisfying \eqref{eq_kinvmetric}. The Iwasawa decomposition yields $\varphi_{t,u}(k) = k_{t,u}a_{t,u}n_{t,u}$ with $a_{t,u} = \exp\rma(t,u,k)$ and $a_{t,u}n_{t,u} \in AN^+(\phi)$. Moreover,%
\begin{equation*}
  \left|\det(\rmd k_{t,u})|_{\fn^+_{\phi}\cdot wb_{\Theta}}\right| = \left|\det(\rmd n_{t,u})|_{\fn^+_{\phi}\cdot wb_{\Theta}}\right| = 1,%
\end{equation*}
because $k_{t,u}$ acts as an isometry and $(\rmd n_{t,u})|_{\fn^+_{\phi}\cdot wb_{\Theta}}$ can be identified with $\Ad(n_{t,u})|_{\fn^+_{\phi}}$, using \eqref{eq_diffeoder}, which is represented by a unipotent matrix in an appropriate basis. It follows that%
\begin{equation*}
  \left|\det(\rmd\varphi_{t,u})|_{\UC_{\Theta,w}(u,x)}\right| = \left|\det(\rmd a_{t,u})|_{\fn^+_{\phi}\cdot wb_{\Theta}}\right| = \left|\det \Ad(a_{t,u})|_{\fn^+_{\phi,w}}\right|.%
\end{equation*}
The eigenvalues of $\Ad(a_{t,u})|_{\fn^+_{\phi,w}}$ are given by $\rme^{\alpha(\rma(t,u,k))}$ with $\alpha\in\Pi^+_{\phi,\Theta,w}$, which implies%
\begin{equation*}
  \left|\det\Ad(a_{t,u})|_{\fn^+_{\phi,w}}\right| = \prod_{\alpha\in\Pi^+_{\phi,\Theta,w}}\rme^{(\dim\fg_{\alpha})\alpha(\rma(t,u,k))} = \rme^{\sigma^+_{\Theta,w}(\rma(t,u,k))} = \rme^{\rma^+_{\Theta,w}(t,u,x)}.%
\end{equation*}
The assertion for $\SC_{\Theta,w}(u,x)$ is proved analogously.%
\end{proof}

\begin{remark}
We note that in the special case that $G = \Sl(d,\R)$ and $\F_{\Theta}$ is the $(d-1)$-dimensional projective space, another characterization of the hyperbolic chain control sets exists, which is not formulated in the Lie-algebraic language, but in terms of the dimensions of the Selgrade bundles of the associated bilinear system on Euclidean space (cf.~\cite[Sec.~7.4]{Kaw}).% 
\end{remark}

\section*{Acknowledgements}

The results in this paper are mainly part of the first author's PhD thesis which was written under the supervision of Luiz San Martin and Fritz Colonius. We thank both of them for numerous mathematical discussions. Furthermore, we thank Anne Gr\"{u}nzig for proof-reading the manuscript. The first author was supported by CAPES grant no.~4229/10-0 and CNPq grant no.~142082/2013-9,
and the second author by DFG fellowship KA 3893/1-1.% 

\end{document}